\documentclass[12pt,a4paper, oneside, reqno]{amsart}

\usepackage[foot]{amsaddr} % authors’ affiliations just below the authors’ names
\usepackage{amsmath, nicefrac, amsthm, verbatim, amsfonts, amssymb, bbm,mathtools}
\usepackage{graphics, xspace, enumerate}
\usepackage[right=25mm, left=25mm, top=25mm, bottom=25mm, marginparsep=2mm, marginparwidth=2cm]{geometry}
\usepackage{marginnote}%this is for simple marginal notes \marginnote{...}
\usepackage[disable]{todonotes}
\usepackage{dsfont}
\usepackage[bb=boondox]{mathalfa}
\usepackage{xcolor}

\usepackage{graphicx, amsmath}
\usepackage[colorlinks=true,citecolor=blue,urlcolor=green,linkcolor=blue,
bookmarksopen=true,unicode=true,pdffitwindow=false, pdfdisplaydoctitle]{hyperref}
\usepackage[english]{babel}
\usepackage[languagenames,fixlanguage]{babelbib}
\usepackage[numbers,sort]{natbib}

\usepackage{seqsplit,cleveref}

\hypersetup{pdfauthor={}}

\theoremstyle{plain}
\newtheorem{theorem}{Theorem}[section]

\newtheorem{lemma}[theorem]{Lemma}
\newtheorem{proposition}[theorem]{Proposition}
\theoremstyle{definition}

\newcommand{\bbR}{\mathbb{R}}

\newcommand{\bbP}{\mathbb{P}}
\newcommand{\bbE}{\mathbb{E}}

\newcommand{\ud}{\mathrm{d}}

\numberwithin{equation}{section}

\title{Stable random walks in cones}

\author[W.\ Cygan]{Wojciech Cygan}
\address[Wojciech Cygan]{
University of Wroc\l{}aw\\
Institute of Mathematics\\
pl.\ Grunwaldzki 2/4\\ 50--384 Wroc\l{}aw, Poland}
\email{wojciech.cygan@uwr.edu.pl}

\author[D.\ Denisov]{Denis Denisov}
\address[Denis Denisov]{Department of Mathematics\\
University of Manchester\\
Manchester, UK}
\email{denis.denisov@manchester.ac.uk}

\author{Zbigniew Palmowski}
\address[Zbigniew Palmowski]{Department of Applied Mathematics\\
Wroc\l{a}w University of Science and Technology\\
Wroc\l{a}aw, Poland}
\email{zbigniew.palmowski@gmail.com}

\author{Vitali Wachtel}
\address[Vitali Wachtel]{Faculty of Mathematics\\
Bielefeld University\\
Bielefeld, Germany}
\email{wachtel@math.uni-bielefeld.de}

\subjclass[2010]{
Primary 60G50; secondary 60G40, 60F17}
\keywords{Random walk, exit time, harmonic function, conditioned process, quasi-stationary distribution}

 \thanks{
    D. Denisov was supported by a Leverhulme Trust Research Project Grant  RPG-2021-105.
    Z. Palmowski acknowledges that the research is partially supported by Polish National Science Centre Grant No. 2023/51/B/ST1/01270.
        V. Wachtel was supported by the Deutsche Forschungsgemeinschaft (DFG, German Research
Foundation)—Project ID 317210226—SFB 1283.
    }

\begin{document}
\allowdisplaybreaks[4]

\begin{abstract}
In this paper we consider a multidimensional random walk killed on  leaving a right circular cone
with a distribution of increments belonging to the normal domain of attraction of
an $\alpha$-stable and rotationally-invariant law with $\alpha \in (0,2)\setminus \{1\}$.
Based on  \cite{MR3771748} describing the tail behaviour of the exit time of $\alpha$-stable process from a cone and using some properties of Martin kernel of the isotropic $\alpha$-stable process, in this paper
we construct a positive harmonic function of the discrete time random walk under consideration. Then we
find the asymptotic tail of the distribution of the exit
time of this random walk from the cone. We also prove the corresponding conditional
functional limit theorem.
\end{abstract}

\maketitle

\section{Introduction}
\subsection{Main results}
It is known that fluctuation theory serves as a powerful machinery to analyse and understand the asymptotic behaviour of one-dimensional random walks. Accompanied by the distinguished Doob $h$-transform method this approach attracted much of attention over few last decades and allowed one to study random walls conditioned to stay positive. Most relevant for the current work are the papers concerned with invariance principles, such as \cite{Iglehart}, \cite{Bolthausen}, \cite{Doney-conditional} and \cite{Caravenna}. As the positive half-line is a prototype of a \textit{cone} in one dimension, it was a natural further research task to focus on random walks conditioned to stay in a cone, or other unbounded domains in higher dimensions. Since fluctuation identities are not applicable in the multidimensional setting, there was need for a novel approach, which was recently developed in  \cite{MR3342657} for random walks having finite second (or higher) moments. In the present work we aim to understand the asymptotic behaviour of random walks that have infinite second moments and are constrained to stay in a circular cone.

More precisely,
we consider a random walk $\{S(n):\, n\ge1 \}$ in $\bbR^d$ of the form
\begin{align}\label{RW}
S(n) = X(1)+\ldots + X(n),
\end{align}
where $X(1), X(2),\ldots$ are independent copies of a generic random vector $X$
%=(X_1,\ldots ,X_d)$
with
the absolutely continuous law with the density $p_X(y)$.
We assume that $X$ belongs to the normal domain of attraction of
an $\alpha$-stable and rotationally-invariant law $p_Z(y)\ud y$ in $\bbR^d$, for $\alpha \in (0,2)\setminus \{1\}$.
That is, the function $p_Z(y)$ is density of $Z(1)$ for the isotropic $\alpha$-stable process $\{Z(t): t\ge0\}$.

The assumption that $X$ belongs to the normal domain of attraction means that a properly rescaled walk converges in distribution towards the process $Z$, that is,
\begin{equation}
\label{eq:weak-conv}
\frac{S(nt)}{n^{1/\alpha}}\Rightarrow Z(t).
\end{equation}
In particular, since there is no centering in relation \eqref{eq:weak-conv}, we have to assume that $\bbE X(1) = \bbE Z(1) =0$ in the case when $\alpha>1$.
We additionally assume  that for some $\nu>0$,
\begin{align}\label{eq:Ass-densities-diff}
\left\vert p_X(y)-p_Z(y)\right\vert \le C \frac{1}{(1+|y|)^{d+\alpha +\nu}},\quad y\in \bbR^d.
\end{align}
For an angle $\theta \in (0,\pi)$ we denote by $K= \{ x\in \bbR^d : \measuredangle{(x,e_d)}<\theta \}$ the \textit{right circular cone} of aperture $\theta$, where $e_d=(0,0\ldots ,0,1)$.
Let
\begin{align}\label{exit-time-cone}
\tau_x = \inf \{ n\ge1: x + S(n) \notin K \}
\end{align}
be the first exit time from a cone $K$ of the random walk $x+S(n)$, which starts at $x\in K$.

We investigate the asymptotic behaviour of
$\bbP(\tau_x>n) $
and the random walk conditioned to stay in cone $K$.
There is a considerable amount of research on the asymptotic behaviour of  random walks with
\emph{finite second moments} in cones and other Lipszhitz domains
which we  briefly revise in Section~\ref{sec:lit_rev} below.
To the best of our knowledge, the present article is the first work where the
asymptotic behaviour of multidimensional random walks
with infinite variance in unbounded domains is analysed.
The methodology that we have developed to treat the random walks in the domain of attraction
of $\alpha$-stable processes is quite general and we
hope that it will be applicable in a number of other situations.
One of the difficulties of the $\alpha$-stable case is that
one cannot assume existence of additional moments as it is usually done in the finite variance case.
Instead, we enforce the local (closeness) assumption displayed at \eqref{eq:Ass-densities-diff}, cf.\ also Subsection~\ref{rem:local}.

A key object in our investigations is the so-called Martin kernel with pole at infinity of the cone $K$ related to the isotropic $\alpha$-stable process $Z(t)$ in $\bbR^d$.
Its existence, uniqueness and main properties were investigated in the seminal paper \cite[Theorem 3.2]{Banuelos-Bogdan-2004} where among other things it was proved that there exists a unique function $M\colon \bbR^d \to [0,\infty )$ such that $M\equiv 0$ on $K^c$ and, for any open and bounded set $B\subset K$,
\begin{align}\label{M-harm}
M(x) = \bbE_x [M(Z(\tau_B^Z))],
\end{align}
where $\tau_B^Z$ is the first exit time from $B$ of the $\alpha$-stable process $Z(t)$.
Identity \eqref{M-harm} means that $M$ is $\alpha$-harmonic in $K$.
The function $M$ is locally bounded and homogeneous of order $\beta = \beta (K, \alpha) \ge0$, that is,
\begin{align}\label{M-homog}
M(x) = |x|^\beta M(x/|x|),\quad x\neq 0.
\end{align}
The number $\beta$ is called \textit{the index of the cone} $K$. We have $\beta =0$ if $K^c$ is a polar set for $Z(t)$ and $0<\beta <\alpha$ otherwise.

We shall first formulate the two main results of the article and only later we will discuss their connections with other existing related results that appeared in the literature.
Our first result concerns the asymptotic behaviour of the tails of the exit time $\tau_x$.
We will write $h(n)\sim g(n)$, as $n\to \infty$, if $\lim_{n\to\infty} h(n)/g(n)=1$.
\begin{theorem}\label{mainthm1}
Assume that the generic increment $X$ of the random walk $S(n)$
has an absolutely continuous law with the density belonging to the domain of attraction of
an $\alpha$-stable and rotationally-invariant law for $\alpha \in (0,2)\setminus \{1\}$.
Assume also that \eqref{eq:Ass-densities-diff} holds true
and that $\bbE X =0$ if $\alpha >1$.
Then
the function
\begin{equation}\label{harmonicfunction0}
V(x)=\lim_{n\to\infty}\bbE[M(x+S(n));\tau_x>n]\end{equation}
is well-defined and is harmonic
for the walk $S(n)$ killed at exiting from the circular cone $K$, that is,
$$
V(x)=\bbE[V(x+S(1));\tau_x>1],\quad x\in K.
$$
Furthermore, for every $x\in K$,
\begin{align}\label{eq:tau-asymp-main}
\bbP(\tau_x>n)\sim \varkappa\frac{V(x)}{n^{\beta/\alpha}},\quad \mathrm{as}\ n\to \infty,
\end{align}
for some constant $\varkappa>0$.
\end{theorem}
In fact, the constant $\varkappa$ from \eqref{eq:tau-asymp-main} is the same as the constant given in equation (3.3) of \cite{MR3771748} and it is determined in terms of the Kelvin transform of the Martin kernel $M$.
To be more precise, we recall the definition of
the heat kernel of the process $Z(t)$ killed  on leaving cone $K$ which is given by
\begin{align}\label{eq:killed-heat-kernel}
p_K (t,x,y) = p_Z(t,x,y) - \bbE \big[p_Z(t-\tau_K, Z(\tau_K), y);\, \tau_K^Z <t\big],\quad x,y\in \mathbb{R}^d,\ t>0,
\end{align}
where $\tau_K^Z$ is the first exit time of the process $Z(t)$ from $K$.
The related killed Green function is defined as
\begin{align}\label{eq:killed-Green-stable}
G_K (x,y) = \int_0^\infty p_K (t,x,y)\, \ud t.
\end{align}
The constant $\varkappa$ is then computed as follows
\begin{align*}
\varkappa = \left( \lim_{K \ni x\to 0}
\frac{G_K (x,e_d)}{M(x)} \right)
 \int_K \int_K
|y|^{\alpha -d} M(y/|y|^2) p_K (1,y,z)\nu_K (z)\, \ud z\, \ud y,
\end{align*}
where $e_d=(0,\ldots ,0,1)\in \bbR^d$; $\nu_K (z) = \int_{K^c}\nu (z-y)\, \ud y,$ and $\nu (y) $ stays for the density of the L\'{e}vy jump measure of the process $Z(t)$.

The second main result concerns the conditional
functional limit theorem for the random walk conditioned to stay in cone $K$.

\begin{theorem}\label{mainthm2}
Under assumptions of Theorem \ref{mainthm1}, the following stochastic process
\[%\bbE
\left(\frac{x+S(nt)}{n^{1/\alpha}}\, \Big| \, \tau_x>n\right)_{t\in [0,1]}
\]
converges weakly in the Skorohod space $D[0,1]$ to a process $me_K(t)$ which we shall call the meander of the $\alpha$-stable process $Z(t)$ in the cone K.
\end{theorem}
The existence of the limiting process is 'almost known' in the literature. More precisely, one knows from \cite[Theorem 3.3]{MR4415390} that the Doob $h$-transform  of $Z$ killed at leaving $K$ has a weak limit as the starting point $x$ converges to zero. Combining this with the entrance law obtained in \cite{MR3771748}, one can easily obtain the existence of the meander. We perform the needed calculations in Subsection \ref{existencemeander} preceding the proof of Theorem~\ref{mainthm2}.

\subsection{Comments on the assumptions and the method}\label{rem:local}
    Assumption \eqref{eq:Ass-densities-diff} on the closeness between the densities should be understood as a control of the rate of convergence in \eqref{eq:weak-conv}. In the case of random walks with finite variance one can quantify the rate of convergence towards the Brownian motion by assuming finiteness of higher moments. Since this is not possible in the case $\alpha<2$, we impose \eqref{eq:Ass-densities-diff}. Assuming some closeness between distributions of $X$ and $Z(1)$ is rather standard in the area of asymptotically stable random
    walks, see e.g.\ the book by Christoph and Wolf~\cite{MR1202035}.
    Assumptions about the local behaviour
    of densities are usually imposed
    while analysing the behaviour of
    the Green function in the infinite variance case, see for example
    ~\cite[Theorem 4]{DenisovWachtel24}.
    In fact, if these assumptions are not in place, the behaviour of the Green function
    can be quite irregular as demonstrated
    in~\cite{Williamson68}.
    Since our method is based on analysis of harmonic functions and
    the Green function of the process $Z$, we had to assume that densities are close to each other.
    It seems reasonable to expect that the geometric assumption that
    the cone is right circular can be relaxed.
    We have imposed this assumption since  the estimates for the killed Green function of the $\alpha$-stable processes and for the Martin kernel $M$
    that are currently available in the literature
    are proved only for such cones.
    It is thus justified to forebode that these estimates should  hold in more general cones. We decided, however, not to consider this question in the present paper to keep it less technical
and to concentrate on the discrete-time framework.

Our  strategy stems from the asymptotics~\eqref{eq:tau-asymp-main},
which factorises in the product of $V(x)$ and $n^{-\beta/\alpha}$.
We thus find first positive harmonic function
$V(x)$  and then use it to  study the tails of exit times and conditional
limit theorems.
Similar strategy
was used in~\cite{MR4718377}, where firstly harmonic function
was found differently using recursive  estimates and repulsion from boundaries
and secondly repulsion from boundaries with KMT (Koml\'{o}s-Major-Tusn\'{a}dy)
coupling
were used to transfers asymptotic results for Brownian motion to random walks.

For random walks with bounded increments in~\cite{MR1643806} firstly a
related approach
was used to find suitable sub/superharmonic functions and then
secondly machinery of Dirichlet forms was used
together with these functions
to obtain estimates for the tail probabilities $\bbP(\tau_x>n)$.

We give now a high level description of the methodology used in the paper.
In the first step we prove existence of harmonic function.
For that we start off with bounding  the  error function
$$f(y)=\bbE[M(y+X)]-M(y)$$ with
$\frac{\Lambda(x)}{(1+\delta(x))^{\varepsilon/2}}$, where $\delta(x):={\rm dist}(x,\partial K)$ and 
\begin{align*}
    \Lambda (x) = \frac{M(x)}{(1+\delta (x))^{\alpha +\varepsilon/2}}.
\end{align*}
We then introduce an error compensator
\begin{align}\label{ul}
    U_\Lambda (x) = \int_K G_K(x,y)\Lambda(y)\ud y,
\end{align}
where the killed Green function $G_K(x,y)$ is given in \eqref{eq:killed-Green-stable}.
This allows us to construct the following superharmonic function
\begin{equation*}
    W(x)=M(x)+U_\Lambda(x).
    \end{equation*}
As a result, we
prove that for sufficiently large $R$ and some $x_0 \in K$,
\begin{align*}
Y_n^{(c)}
&=W(x+Rx_0+S(n))
{\rm 1}\{\tau_x>n\}\\
&\hspace{3cm}
+c\sum_{k=0}^{n-1}\Lambda(x+Rx_0+S(k)){\rm 1}\{\tau_x>k\},
\end{align*}
is a supermartingale.
Since the error function $f$ is bounded from above by $\Lambda$,
the existence of the harmonic function $V(x)$ is a consequence of the martinagle property of
$$
L_n=M(x+Rx_0+S(n\wedge\tau_x))
-\sum_{k=0}^{n\wedge\tau_x-1}f(x+Rx_0+S(k)),\quad n\ge0,
$$
the optional stopping theorem and the dominated convergence theorem,
together with bounds that follow from the supermartingale $Y_n^{(c)}$.

In the second step we then proceed to study of tails and conditional limit theorems.
Here, we first obtain sharp upper bound for $\bbP (\tau_x>n)$
and then probabilities of large deviations by using recursive
bounds. These bounds together with functional central limit theorems
deliver the final results.

In our proof the main
difficulty is in proving the supermartingale property of $Y_n^{(c)}$.
To do so we have to show completely new estimates in the context
of the $\alpha$-stable processes and random walks with increments
distribution belonging to the $\alpha$-stable law with $\alpha \in (0,2)\setminus \{1\}$.
In particular, in Lemma \ref{lem:error-estimate} we find the estimates of the error function $f$
and in Lemma \ref{lem:Green-bounds-our} we identify the
upper bound of the killed Green function $G_K(x,y)$.

The main novelty and difficulty are in showing the estimates related
to $U_\Lambda (x)$ given in \eqref{ul}. In particular, a completely new approach is required in finding
the proper upper bound for the gradient of $U_\Lambda(x)$ (see Lemma \ref{gradientestimate}),
for $|U_\Lambda (x+y) - U_\Lambda (x)|$
if $0<\alpha <1$ (see Lemma \ref{lem:U-Lambda-diff-est})
and for $|U_\Lambda (x+y) - U_\Lambda (x) -\nabla U_\Lambda (x)\cdot y|$
if $1< \alpha <2$ (see Lemma \ref{lem:U_Lambda-diff-alpha>1}).

We also proved some other facts that are of own interest, like the existence of the meander $m_K(t)$ (see Theorem \ref{mainthm2}).

\subsection{Literature overview}\label{sec:lit_rev}
The asymptotic behaviour of random walks  with finite variance
in cones  has been
extensively studied in the literature.
The most general result was obtained by
Denisov and Wachtel \cite{MR4718377}
who proved counterparts of Theorems \ref{mainthm1} and \ref{mainthm2}
under additional moment condition $\bbE [|X|^2\log(1+|X|)]<\infty$ which evidently does not fit the scope of the present article.
See also \cite{MR3342657, DenisovWachtel19} for similar result but proved using different method under stronger moment conditions.
In these studies an important role
was played by the KMT coupling of Brownian motion
and random walks with finite variance.
This tool is not available in the infinite-variance case, which required a new approach.
Markov chains in cones have been studied recently in~\cite{DenisovZhang23},
where also Functional Central Limit Theorem was used instead of the KMT coupling.
Raschel and Tarrago in \cite{MR4243159} also studied the asymptotic behaviour of zero-drift random walks
confined to multidimensional convex cones but they focused on the case when the endpoint is close to the boundary and
they derived a local limit theorem in the fluctuation regime.
Invariance principle was obtained in \cite{MR4102254}, which corresponds to our Theorem \ref{mainthm2}, but in this case a Brownian meander is appearing in the limit.
See also \cite{MR4443298} for the asymptotic behaviour of the Green function $G(x,y)$ of the killed random walk on exiting from the cone where $y$ tends to infinity in different ways.

The research question posed in this paper has gained lots of attention in various related contexts.
The asymptotic tail of the distribution of the exit time from a cone has been studied first for Brownian motion,  see \cite{MR0863716, MR1465162}.
For one dimensional L\'evy processes and half-line it has been studied in \cite{MR1303922, MR2248228}. For the result concerning one-dimensional
self-similar processes and half-line see \cite{MR2971725}.
The multivariate case of $\alpha$-stable processes has been in analysed in \cite{MR3771748, MR4415390}.
One can derive similar result under additional assumption that multidimensional stochastic process has a skew-product structure; see \cite{MR4329614} for details.

Strongly related with the results concerning the continuous-time stochastic processes are discrete-time counterparts.
For more general cones though and essentially for random
walks with increments having a finite support, Varopoulos in \cite{MR1643806} gave lower and
upper bounds for the probability $\bbP(\tau_x>n)$.
Behaviour of the Green function of asymptotically stable random walk
on a half-space has been recently considered in~\cite{DenisovWachtel24}.

Weak convergence of a normalized two-dimensional random walk conditioned to stay in a cone
to the analogous conditional distribution of a two-dimensional Wiener process was established in \cite{MR1137264}.
In \cite{MR2674995} a complete representation of the Martin boundary of killed random walks on the quadrant $\mathbb{N}^*\times \mathbb{N}^*$ was derived.

Random walks with drift have been studied in  \cite{MR3163211, MR3283611, MR3512425}.
Different approach is taken in \cite{MR4636896}, where for a random walk having integrable increments with drift in the interior of the closed cone, the asymptotics of $\bbP(\tau_x>n)$ is derived. It is in the form
of the sum of survival probability and a certain sequence $\rho^n B_n$ for some $\rho\in (0,1)$ and $B_n \to 0$ satisfying $B_n^{1/n}\to 1$ as $n\to+\infty$.

The Weyl chamber of the form $K=\{x\in \mathbb{R}^d: x_1<x_2<\ldots<x_d\}$ is a specific cone that
attracted a lot of attention.
In this case there are many papers related to the probability $\bbP(\tau_x>n)$; see for example \cite{DenisovWachtel10, MR1678525, MR2176549, MR2438702}
and references therein.

Finally, the probability  $\bbP(\tau_x>n)$ can determine so-called persistency which corresponds to identifying
the rate of decay of this probability; see \cite{DenisovWachtel10, MR3342657, MR3315616, MR3780696, MR3468226} for details and other related references.

\subsection{Organisation of the paper}
The paper is organised as follows.
In Section \ref{sec:prel} we collect and prove the key estimates of the Martin kernel $M(x)$, of the so-called error function $f(x) = \bbE M(x+X) - M(x)$,  and of
the killed Green function denoted by $G(x,y)$. In Section \ref{sec:supermartingale} we construct the harmonic function $V(x)$ defined in \eqref{harmonicfunction0}.
Using this result, in Section \ref{sec:asymptoticstau}, we find the asymptotic tail of the distribution of $\tau_x$, cf.\ Theorem \ref{mainthm1}.
In the closing Section \ref{sec:copnditionallimit} we display the proof of the conditional
functional limit theorem stated in Theorem \ref{mainthm2}.

\section{Preliminaries}\label{sec:prel}

\subsection{Estimates of the Martin kernel}
We recall that by \eqref{M-harm} the function $M$ is harmonic in the circular cone $K$ for the isotropic $\alpha$-stable process $Z(t)$ killed on the exiting from $K$.
Throughout this paper we assume that $\alpha \in (0,2)\setminus \{1\}$
and write $C$ or $c$ with some possible indexes to denote some constants.
By \cite[Lemma 3.3]{Michalik-2006}, we have
\begin{align}\label{eq:M-Michalik-up}
M(x)\le C  |x|^{\beta-\frac{\alpha}{2}}\delta(x)^{\alpha /2},\quad x\in K,
\end{align}
and
\begin{align}\label{eq:M-Michalik-low}
M(x)\ge c  |x|^{\beta-\frac{\alpha}{2}}\delta(x)^{\alpha /2},\quad x\in K,
\end{align}
where \[\delta (x) = \mathrm{dist}(x,\partial K).\]
Hence, there is a constant $C>0$ such that
\begin{align}\label{eq:M-beta-bound}
M(x)\le C|x|^\beta,\quad x\in \bbR^d.
\end{align}
From \eqref{eq:M-beta-bound} we can immediately conlcude the following lemma.
\begin{lemma}\label{lem:M-bounds}
%Let $M$ be the Martin kernel of a cone $K$.
It holds
\begin{align}\label{eq:M-bound-1}
|M(x+y) - M(x)|\le C \left( |x|\vee |y|\right)^\beta,\quad x,x+y\in K.
\end{align}
\end{lemma}
We will also need a slightly different estimate.
\begin{lemma}
\label{lem:M-newbound}
There exists a constant $C$ such that
$$
M(x+y)\le C\left(M(x)+|x|^{\beta-\alpha/2}|y|^{\alpha/2}\right),
\quad x,x+y\in K,\ |y|\le|x|.
$$
\end{lemma}
\begin{proof}
Using \eqref{eq:M-Michalik-up} and recalling that
$\beta-\alpha/2<\alpha/2<1$, we have
\begin{align*}
M(x+y)
&\le C|x+y|^{\beta-\alpha/2}\delta(x+y)^{\alpha/2}\\
&\le C\left(|x|^{\beta-\alpha/2}+|y|^{\beta-\alpha/2}\right)\left(\delta(x)^{\alpha/2}+|y|^{\alpha/2}\right)\\
&\le 2C |x|^{\beta-\alpha/2}
\left(\delta(x)^{\alpha/2}+|y|^{\alpha/2}\right).
\end{align*}
Applying now \eqref{eq:M-Michalik-low} we get the desired estimate.
\end{proof}

The following estimate for the gradient of $M$ follows directly from \cite[Lemma 3.2]{Bogdan-Kulczycki-Nowak-2002}.
\begin{lemma}
%Let $M$ be the Martin kernel of a cone $K$.
We have
\begin{align}\label{eq:M-gradient-0}
|\nabla M(x)|\le C \frac{M(x)}{\delta (x)},\quad x\in K.
\end{align}
\end{lemma}
We need another new estimate related to $M(x)$.
\begin{lemma}\label{lem:M-gradient}
%Let $M$ be the Martin kernel of a circular cone $K$. Then,
For any $0< \eta <1 $ and any $x\in K$ with $\delta(x)>1$ and  $|y|\le\delta(x)^{1-\eta}$ it holds %$0<r<\delta(x)$,
\begin{align}\label{eq:M-gradient}
|M(x+y)-M(x) - \nabla M(x)\cdot y|\le C\, |y|^2\frac{M(x)}{\delta(x)^2}.
\end{align}
\end{lemma}

\begin{proof}[Proof of Lemma \ref{lem:M-gradient}]
We first find an estimate for the second derivative of $M$.
We make use of the following representation (see e.g.\ \cite{Bogdan-Kulczycki-Nowak-2002}) for the function $M$.
For any $z_0\in K$ and any $r>0$ such that $B_r(z_0)\subset K$, it holds
\begin{align}\label{eq:M-Poisson}
M(z) = \int_{B_r(z_0)^c} P_r (z-z_0, w-z_0) M(w) \,\ud w,\quad z\in B_r(z_0),
\end{align}
where by $P_r (\theta,w)$ we denote the Poisson kernel of the centred ball $B_r(0)$ given by
\begin{align*}
P_r (\theta,w) =
\frac{\Gamma(d/2)\sin(\pi \alpha /2)}{\pi^{d/2+1}}
\left( \frac{r^2-|\theta|^2}{|w|^2 - r^2} \right)^{\alpha /2}\!\!
|w-\theta|^d,\quad \mathrm{for}\ |\theta|<r\ \mathrm{and}\  |w|>r.
\end{align*}
Routine calculations yield
\begin{align*}
\nabla_{\theta_i}P_r(\theta,w) = P_r(\theta,w)F_i(\theta,w),\quad \mathrm{where}\
F_i(\theta,w) = \frac{-\alpha \theta_i}{r^2-|\theta|^2}+d\frac{w_i-\theta_i}{|w-\theta|^2}.
\end{align*}
From this we easily derive formulas for the second derivatives:
\begin{align}\label{eq:Poiss-deriv-1}
\begin{split}
&D^2_{\theta_i \theta_i} P_r(\theta,w)
\\
&=
P_r(\theta,w) \left\{ F^2_i(\theta,w)
-
\frac{\alpha (r^2-|\theta|^2)+2\alpha \theta_i^2}{(r^2-|\theta|^2)^2}
+
d\,
\frac{2(\theta_i-w_i)^2-|w-\theta|^2}{|w-\theta|^4}
\right\}
\end{split}
\end{align}
and
\begin{align}\label{eq:Poiss-deriv-2}
\begin{split}
&D^2_{\theta_i \theta_j} P_r(\theta,w)\\
&=
P_r(\theta,w) \left\{ F_i(\theta,w)F_j(\theta,w)
-
\frac{2\alpha \theta_i \theta_j}{(r^2-|\theta|^2)^2}
-
\frac{2d(w_i-\theta_i)(\theta_i - w_i)}{|w-\theta|^4}
\right\}.
\end{split}
\end{align}
Combining identity \eqref{eq:M-Poisson} %applied to $x=x_0$ and $x=z$
with equation \eqref{eq:Poiss-deriv-1}  implies that for any $z\in K$ such that $ B_r(z)\subset K$ we have
\begin{align*}
&D^2_{z_i z_i}M(z) \\
&=
\int_{B_r(z)^c}
M(w)
P_r(0,w-z)
\left\{ F_i^2(0,w-z) - \frac{\alpha }{r^2} +\frac{2d(w_i-z_i)^2-|w-z|^2}{|w-z|^4}\right\}\ud w.
%\\
%&\le C M(z)r^{-2}.
\end{align*}
Further, we easily verify that $|F_i(0,w-z)|\le d/|w-x|$ and whence, for any $z\in K$ and any $r>0$ such that $B_r(z)\subset K$,
\begin{align}\label{eq:M-sec-deriv-1}
|D^2_{z_i z_i}M(z)| &
\le C
\frac{M(z)}{r^2}
\stackrel{r\to\delta(z)}{\longrightarrow}C\frac{M(z)}{\delta^2 (z)},
\quad i=1,\ldots ,d.
\end{align}
With the aid of \eqref{eq:Poiss-deriv-2} we can similarly show that for any $z\in K$ and any $r>0$ such that $B_r(z)\subset K$,
\begin{align}\label{eq:M-sec-deriv-2}
|D^2_{z_i z_j}M(z)|\le C\frac{M(z)}{r^2}\stackrel{r\to\delta(z)}{\longrightarrow}C\frac{M(z)}{\delta^2 (z)}, \quad i\neq j,\ i,j=1,\ldots ,d.
\end{align}
We now turn to the proof of \eqref{eq:M-gradient}. Let $x\in K$ and let $\delta= \delta (x)>1$. We choose $\rho>0$ such that $B_\rho(x) \subset K$.
By the Taylor formula, for $|y|\le\delta^{1-\eta}$ we have
\begin{align*}
|M(x+y)-M(x) - \nabla M(x)\cdot y|
\leq
\max_{z\in B_\rho (x)}\max_{1\le i,j\le d}|D^2_{z_i z_j}M(z)|\, |y|^2.
\end{align*}
By equations \eqref{eq:M-sec-deriv-1} and \eqref{eq:M-sec-deriv-2} we obtain
\begin{align*}
|M(x+y)-M(x) - \nabla M(x)\cdot y|\le C |y|^2\max_{z\in B_\rho (x)}\frac{M(z)}{\delta(z) ^2},\quad \mathrm{for}\ x\in K,\, |y|\le\delta^{1-\eta}.
\end{align*}
We next take $\rho=\delta^{1-\eta}$ and observe that then for any $z\in B_{\delta ^{1-\eta}}(x)$ we clearly have $c\delta \le\delta (z) \le2\delta $ and $\big||z|-|x|\big|\le\delta^{1-\eta}$. In view of \eqref{eq:M-Michalik-up} and \eqref{eq:M-Michalik-low} this implies that
\begin{align*}
M(z)\le C |z|^{\beta -\alpha /2}\delta(z)^{\alpha/2}\le C_1|x|^{\beta -\alpha/2}\delta^{\alpha /2}\le C_2 M(x)
\end{align*}
and we arrive at
\begin{align*}
\max_{z\in B_{\delta^{1-\eta}} (x)}\frac{M(z)}{\delta(z) ^2}\le C \frac{M(x)}{\delta^{2}},
\end{align*}
which  completes the proof.
\end{proof}

\subsection{Error estimates}

We introduce the following \textit{error} function
\begin{align}\label{errorfunction}
f(x) = \bbE M(x+X) - M(x),\quad x\in K.
\end{align}
%where $X$ is an independent copy of the first jump $X(1)$.

\begin{lemma}\label{lem:error-estimate}
There exists $\varepsilon >0$ such that
\begin{align}
|f(x)|\le C \frac{M(x)}{\delta (x)^{\alpha +\varepsilon}},\quad x\in K\text{ with }\delta(x)\ge1.
\end{align}
\end{lemma}

\begin{proof}
Recall that the function $M$ is harmonic for $Z(t)$ in $K$.
We denote \begin{equation}\label{Zand tauZ}Z=Z(1),\quad \tau^Z_K=\inf\{t\geq 0: Z(t)\notin K\}.\end{equation}
We start by rewriting the function $f(x)$ as follows
\begin{align}\label{eq:I-J-decomp}
\begin{split}
f(x) &= \bbE M(x+X) - \bbE[ M(x+Z);\, \tau^Z_K> 1]\\
&=
\bbE \left[ M(x+X) - M(x+Z)\right] + \bbE[ M(x+Z);\, \tau^Z_K\le1]\\
&=
\int \left( M(x+y) - M(x)\right) \left( p_X(y) -p_Z(y)\right) \ud y\\
&\qquad \qquad + \bbE[ M(x+Z);\, \tau^Z_K\le1]
=: I + J.
\end{split}
\end{align}
%where $\tau^Z_K$ denotes the first exit time of the process $Z(t)$ from the cone $K$.
We shall first handle the integral $I$  in the case when $\alpha <1$. We have
\begin{align*}
I  &=\left( \int_{|y|<\delta^{1-\eta}} + \int_{\delta^{1-\eta} <|y|} \right)
 \left( M(x+y) - M(x)\right) \left( p_X(y) -p_Z(y)\right) \ud y
 =:
 I_1 +I_2,
\end{align*}
where we set again $\delta = \delta (x)$.
We start by estimating $I_1$. We use the fundamental theorem for line integrals in the form
\begin{align*}
M(x+y) - M(x) = \int_{0}^1 \nabla M(x+\varrho y)\cdot y\, \ud \varrho.
\end{align*}
We combine this formula with assumption \eqref{eq:Ass-densities-diff} and we obtain that for some $\varrho_0 \in (0,1)$,
\begin{align*}
|I_1|\le C \int_{|y|<\delta^{1-\eta}}|\nabla M(x+\varrho_0y)||y| \frac{1}{(1+|y|)^{\alpha +d+\nu}}\ud y.
\end{align*}
We note that for $|y|\le\delta^{1-\eta}$ we easily check that
$c_1\delta (x)\le\delta(x+\theta y)\le c_2\delta(x)$ and $|x+\theta y|\le c_3|x|$, for some constants $c_1,c_2,c_3>0$. These estimates combined with \eqref{eq:M-gradient-0} and \eqref{eq:M-Michalik-up}-\eqref{eq:M-Michalik-low} imply that
\begin{align*}
|\nabla M(x+\varrho_0 y)|\le\frac{M(x+\varrho_0 y)}{\delta (x+\varrho_0 y)}\le C\frac{M(x)}{\delta (x)},
\end{align*}
and whence,  %by \eqref{eq:M-beta-bound},
\begin{align}\label{eq:I_1-alpha<1}
|I_1|\le C\frac{M(x)}{\delta} \int_{0}^{\delta^{1-\eta}}r^{-\alpha - \nu}\ud r = C \frac{M(x)}{\delta^{\alpha + \nu (1-\eta) + \eta (1-\alpha)}}
%\leq
%C_1
%\frac{|x|^\beta}{\delta^{\alpha +\nu (1-\eta) + \eta (1-\alpha)}}
,
\end{align}
where $\nu>0$ from \eqref{eq:Ass-densities-diff} is chosen in such a way that $\alpha +\nu <1$.
Note that it can be always done.
To estimate integral $I_2$ we split it into two parts and apply Lemmas~\ref{lem:M-newbound} and \ref{lem:M-bounds} which yields
\begin{align}\label{eq:I_2-est}
\begin{split}
|I_2|&\leq
\left( \int_{\delta^{1-\eta} <|y|<|x|} + \int_{|y|>|x|}\right)
 \left\vert M(x+y) - M(x)\right\vert \left\vert p_X(y) -p_Z(y)\right\vert \ud y\\
 &\leq
 C_1M(x) \int_{\delta^{1-\eta} <|y|<|x|} \frac{\ud y}{(1+|y|)^{\alpha +d+\nu}}
 +
 C_1|x|^{\beta-\alpha/2}
 \int_{\delta^{1-\eta} <|y|<|x|}
 \frac{|y|^{\alpha/2}\ud y}{(1+|y|)^{\alpha+d+\nu}}\\
 &\hspace{1cm}+
 C_2
 \int_{|y|>|x|} |y|^\beta \frac{\ud y}{(1+|y|)^{\alpha +d+\nu}}\\
 &\leq
 C_3 M(x)
 \frac{1}{\delta^{(1-\eta)(\alpha +\nu)}}
 +
 C_4|x|^{\beta-\alpha/2}\frac{1}{\delta^{(1-\eta)(\alpha +\nu)-\alpha/2}}
 +C_5
 \frac{|x|^\beta}{|x|^{\alpha +\nu}}
 \leq
 C_6
  \frac{M(x)}{\delta^{\alpha +\varepsilon}},
 \end{split}
\end{align}
where we used the inequality $\delta(x)\le |x|$ and \eqref{eq:M-Michalik-low}.

We next handle $I$ in \eqref{eq:I-J-decomp} in the case when $1< \alpha <2 $. We rewrite the integral $I$ as follows
\begin{align*}
I = \int&\left( M(x+y) - M(x) -\nabla M(x)\cdot y\right) \left( p_X(y) -p_Z(y)\right) \ud y\\
&\qquad +
\nabla M(x) \cdot \int y \left( p_X(y) -p_Z(y)\right) \ud y
\end{align*}
and we observe that the second integral vanishes as we have assumed $\bbE X = \bbE Z =0$. We then proceed similarly as before, that is we decompose $I$ into two parts
\begin{align*}
I  &=\left( \int_{|y|<\delta^{1-\eta}} + \int_{\delta^{1-\eta} <|y|} \right)
\left( M(x+y) - M(x) -\nabla M(x)\cdot y\right) \left( p_X(y) -p_Z(y)\right) \ud y\\
&=:
 I_1 +I_2.
\end{align*}
To estimate $I_1$ we can directly apply \eqref{eq:M-gradient} which yields
\begin{align*}
|I_1|&\le C_1\frac{M(x)}{\delta^2}\int_{0}^{\delta^{1-\eta}} |y|^2\frac{\ud y}{(1+|y|)^{d+\alpha +\nu}}\le C_2
\frac{M(x)}{\delta^{\alpha + \nu (1-\eta) + \eta (2-\alpha)}}
%\le C_3 \frac{|x|^\beta}{\delta^{\alpha + \nu (1-\eta) + \eta (2-\alpha)}},
,
\end{align*}
where $\nu$ from \eqref{eq:Ass-densities-diff} is chosen in such a way that $\alpha +\nu <2$.
To estimate $I_2$ we write
\begin{align*}
|I_2|&\leq
 \int_{\delta^{1-\eta} <|y|}
 \left\vert M(x+y) - M(x)\right\vert \left\vert p_X(y) -p_Z(y)\right\vert \ud y\\
&\qquad +
 |\nabla M(x)| \int_{\delta^{1-\eta} <|y|} |y|\left\vert p_X(y) -p_Z(y)\right\vert \ud y
\end{align*}
and we notice that the first integral can be handled exactly in the same manner as in \eqref{eq:I_2-est}, whereas for the second integral we have
\begin{align*}
 |\nabla M(x)| \int_{\delta^{1-\eta} <|y|} |y|\left\vert p_X(y) -p_Z(y)\right\vert \ud y
 &\leq
 C_1
 \frac{M(x)}{\delta}\int_{\delta^{1-\eta}}^\infty |y|\frac{\ud y}{(1+|y|)^{\alpha +d+\nu}}
 \\
 &\leq
 C_2
 \frac{M(x)}{\delta^{\alpha +\nu +\eta (1-\alpha -\nu)}} .
\end{align*}
We observe that we can choose $\nu>0$ from \eqref{eq:Ass-densities-diff} and $0<\eta <1$ such that $\nu +\eta (1-\alpha -\nu)>0$. Indeed, it trivially holds 
\[
\nu +\eta (1-\alpha -\nu)>0 \Longleftrightarrow \nu \left( \frac{1}{\eta}-1\right) > \alpha -1  .
\]
We conclude that
$|I_2|\le C \frac{M(x)}{\delta^{\alpha + \varepsilon}}$, for some $\varepsilon >0$.

We are left to estimate the integral $J$ from \eqref{eq:I-J-decomp}. We start with the following formula
\begin{align*}
J& =
\int_{0}^1 \int_{K^c}\bbE \left[ M(y+Z_{1-s});\, y+Z_{1-s}\in K\right]\bbP_x (\tau \in \ud s,\, Z_s\in \ud y)\\
&\leq
\int_{0}^1 \int_{K^c}\bbE \left[ M(y+Z_{1-s});\, y+Z_{1-s}\in K\right]\bbP_x (Z_s\in \ud y).
\end{align*}
We shall now estimate the expectation in the last integral. In the first step we consider the event $\{|y+Z_{1-s}|\le|y|/2\}$ on which we have the bound $M(y+Z_{1-s})\le C |y|^\beta$, where we used \eqref{eq:M-beta-bound}. It follows that
\begin{align}\label{eq:J-est-1}
\begin{split}
\bbE &\left[ M(y+Z_{1-s});\, y+Z_{1-s}\in K,\, |y+Z_{1-s}|\le|y|/2 \right]\\
&\leq
C |y|^\beta \,
\bbP \left( |y+Z_{1-s}|\le|y|/2\right)
\leq
C |y|^\beta \, \bbP (|Z_{1-s}|\ge|y|/2)\\
&\leq
C_1(1-s)\frac{|y|^\beta}{(1+|y|)^{\alpha}}
\le C_1 (1-s) \frac{|y|^{\beta -\alpha /2}}{(1+\delta (y))^{\alpha /2}},
\end{split}
\end{align}
where in the last inequality we applied the inequality $\delta (y)\le|y|$.
We next consider the event $\{|y+Z_{1-s}|\ge2|y|\}$ on which it evidently holds  $|Z_{1-s}|\ge|y|$. Hence
\begin{align}\label{eq:J-est-2}
\begin{split}
\bbE &\left[ M(y+Z_{1-s});\, y+Z_{1-s}\in K,\, |y+Z_{1-s}|\ge2|y| \right]\\
&\leq C
\bbE \left[ |Z_{1-s}|^\beta ;\, |Z_{1-s}|\ge|y| \right]
\leq
C_1 (1-s)\frac{1}{(1+|y|)^{\alpha - \beta}}\\
&\leq
C_2
(1-s) \frac{|y|^{\beta - \alpha /2}}{(1+\delta (y))^{\alpha /2}}.
\end{split}
\end{align}
We end up considering the event $\{ |y+Z_{1-s}|\in (|y|/2, 2|y|)\}$ and we apply \eqref{eq:M-Michalik-up} to obtain
\begin{align}\label{eq:J-est-3}
\begin{split}
\bbE &\left[ M(y+Z_{1-s});\, y+Z_{1-s}\in K,\, |y+Z_{1-s}|\in (|y|/2, 2|y|) \right]\\
&\leq
C
|y|^{\beta -\alpha/2}
\bbE \left[ \delta (y + Z_{1-s})^{\alpha /2} ;\,  y+Z_{1-s}\in K\, |y+Z_{1-s}|\in (|y|/2, 2|y|) \right] \\
&\leq
C_1
|y|^{\beta -\alpha/2}
\bbE \left[ |Z_{1-s}|^{\alpha /2};\, |Z_{1-s}|>\delta(y)/2 \right]
\le C_2
(1-s) \frac{|y|^{\beta -\alpha/2}}{(1+\delta(y))^{\alpha /2}}.
\end{split}
\end{align}
Next, by combining estimates \eqref{eq:J-est-1}--\eqref{eq:J-est-3} with the bound for the transition function of the stable process
\begin{align}\label{eq:stable-denisty-bound}
p_Z(s,y)  \asymp \frac{s}{1+|y|^{d+\alpha}},\quad s>0,\  y\in \bbR^d,
\end{align}
we obtain
\begin{align*}
J&\le C
\int_0^1 \int_{K^c}
s(1-s) \frac{1}{1+|y-x|^{d+\alpha}}
 \frac{|y|^{\beta -\alpha/2}}{(1+\delta(y))^{\alpha /2}}\,
 \ud y\, \ud s\\
 &= C_1
  \int_{K^c}
\frac{1}{1+|y-x|^{d+\alpha}}
 \frac{|y|^{\beta -\alpha/2}}{(1+\delta(y))^{\alpha /2}}\,
 \ud y
 = C_1
 \int_{K^c-x}\frac{1}{1+|z|^{d+\alpha}}
 \frac{|z+x|^{\beta -\alpha/2}}{(1+\delta(z+x))^{\alpha /2}}\,
 \ud z\\
 &\le C_2
 |x|^{\beta -\alpha/2}\int_{|z|\ge\delta}\frac{1}{1+|z|^{d+\alpha}}\ud z
 +
 C_2
 \int_{K^c-x}\frac{|z|^{\beta -\alpha/2}}{(1+|z|^{d+\alpha})(1+\delta(z+x))^{\alpha /2}}\,
 \ud z\\
& \le C_3 \frac{|x|^{\beta-\alpha/2}}{(1+\delta)^\alpha}
 +
 C_2
  \int_{|z|\ge\delta}\frac{|z|^{\beta -\alpha/2}}{(1+|z|^{d+\alpha})}\,  \ud z
  \le C_3 \frac{|x|^{\beta-\alpha/2}}{(1+\delta)^\alpha}
 +
 C_4(1+\delta)^{\beta -3\alpha/2}\\
 &\leq
 C_5\frac{|x|^{\beta-\alpha/2}}{(1+\delta)^\alpha}
 \leq
 C_6
 \frac{M(x)}{(1+\delta)^{3\alpha /2}},
\end{align*}
where in the last line we used \eqref{eq:M-Michalik-low}.
\end{proof}

%\begin{remark}
%About better bound with $2\alpha$. {\it To do!}
%\end{remark}

\subsection{Killed Green function estimates}
In what follows we will abandon the index $K$ while working with the killed Green function, whence we will write
\begin{equation*}
G(x,y) = G_K (x,y),\quad x,y\in K,
\end{equation*}
where $G_K(x,y) $ is the function defined at \eqref{eq:killed-Green-stable}. According to \cite[Theorem 3.11]{Michalik-2006}, for any circular cone $K$ centred at the origin, and for any $x,y\in K$ it holds\footnote{We write $f(x)\asymp g(x)$ if there are two constants $c,C>0$ such that $c\, g(x)\le f(x)\le C g(x)$.}
\begin{align}\label{eq:Green-f-killed-bounds}
G(x,y) \asymp
\frac{A_{d,\alpha}}{|x-y|^{d-\alpha}}\wedge
\frac{\delta (x)^{\alpha/2}\delta(y)^{\alpha/2}}{|x-y|^d}\left( \frac{|x|\wedge |y|}{|x|\vee |y|}\right)^{\beta -\alpha /2},
\end{align}
where $A_{d,\alpha} = \frac{\Gamma \left( \frac{d-\alpha}{2}\right)}{2^\alpha \pi^{d/2}\Gamma \left( \frac{\alpha}{2}\right)}$. According to \cite[Corollary 3.3]{Bogdan-Kulczycki-Nowak-2002},  the following estimate is valid for the gradient of $G$,
\begin{align}\label{eq:Green-killed-gradient-est}
|\nabla_x G(x,y)|\le C\frac{G(x,y)}{\delta (x)\wedge |x-y|},\quad x,y\in K,\, x\neq y.
\end{align}
Furthermore, since for each $y\in K$ the function $G(\cdot ,y)$ is $\alpha$-harmonic in $K\setminus \{y\}$, proceeding similarly as in the proof of Lemma \ref{lem:M-gradient} and using the Poisson kernel (cf.\ \eqref{eq:M-Poisson}) we can prove that for all $i,j=1,\ldots ,n$,
\begin{align}\label{eq:Green-killed-second-deriv-est}
|D^2_{x_ix_j}G(x,y)|\le C \frac{G(x,y)}{\left( \delta(x) \wedge |x-y|\right)^2},\quad x,y\in K,\, x\neq y.
\end{align}
We will prove now the key estimate of the killed Green function.
\begin{lemma}\label{lem:Green-bounds-our}
For any $A>0$ there is a constant $C=C_A>0$ such that 
\begin{align}\label{eq:Green-f-killed-our-bounds}
G(x,y) \le C
\begin{cases}
\frac{M(x) M(y)}{|y|^{d-\alpha +2\beta}},& |x|\le|y|,\, |x-y|\ge A|x|;\\
\frac{M(x) M(y)}{|x|^{d-\alpha +2\beta}},& |y|\le|x|,\, |x-y|\ge A|x|;\\
\frac{1}{|x-y|^d}
\frac{M(x) M(y)}{|x|^{\beta -\alpha/2}|y|^{\beta -\alpha/2}}, & \frac{\delta(x)}{2}\le|x-y|\le A|x|;\\
\frac{1}{|x-y|^{d-\alpha}},& |x-y|\le\frac{\delta(x)}{2}.
\end{cases}
\end{align}
\end{lemma}

\begin{proof}
We start with the case $|x|\le|y|$ and $|x-y|\ge A|x|$. Then estimate \eqref{eq:Green-f-killed-bounds} yields
\begin{align}\label{eq:Green-est-help}
G(x,y) &\le C \frac{\delta(x)^{\alpha /2} \delta(y)^{\alpha /2}}{|x-y|^d}\left(\frac{|x|}{|y|}\right)^{\beta -\alpha/2}\nonumber \\
&=
C \frac{\delta(x)^{\alpha /2} |x|^{\beta -\alpha/2}\delta(y)^{\alpha /2}  |y|^{\beta -\alpha/2}}{|x-y|^d\, |y|^{2\beta -\alpha }}
\leq
C_1
\frac{M(x) M(y)}{|x-y|^d\, |y|^{ 2\beta-\alpha }},
\end{align}
where in the last line we applied \eqref{eq:M-Michalik-low}.
We then distinguish between two cases. First we assume that $|x|\le|y| \le(1+A)|x|$, then
$|x-y|\ge A|x|\ge\frac{A}{1+A}|y|$ and
we obtain the desired bound directly from \eqref{eq:Green-est-help}.
Next, we assume $|y|>(1+A)|x|$, and then
$|x-y|\ge|y|-|x|\ge\frac{A}{1+A}|y|$,
 so the final estimate follows again from \eqref{eq:Green-est-help}.

The case when $|y|\le|x|$ and $|x-y|\ge A|x|$ can be handled in the same way as the previous one.

In the case when
$\frac{\delta(x)}{2}\le|x-y|\le A|x|$, we combine \eqref{eq:Green-f-killed-bounds} with \eqref{eq:M-Michalik-up} and we obtain
\begin{align*}
G(x,y)&\le C
\frac{\delta (x)^{\alpha/2}\delta(y)^{\alpha/2}}{|x-y|^d}\le C_1
\frac{1}{|x-y|^d}
\frac{M(x) M(y)}{|x|^{\beta -\alpha/2}|y|^{\beta -\alpha/2}}.
\end{align*}
The last bound for $|x-y|<\frac{\delta(x)}{2}$ follows directly from \eqref{eq:Green-f-killed-bounds} if we drop the second part of the minimum appearing in that estimate.
\end{proof}

\section{Construction of a non-negative supermartingale}
\label{sec:supermartingale}

We start by investigating the following function
\begin{align}\label{eq:U-lambda-f}
U_\Lambda (x) = \int_K G(x,y)\Lambda(y)\ud y,\quad \text{for\ } x\in K,
\end{align}
where
\begin{align}\label{Lambda}
\Lambda (y) = \frac{M(y)}{(1+\delta (y))^{\alpha +\varepsilon/2}},\quad y\in K,
\end{align}
and $\varepsilon >0$ comes from Lemma \ref{lem:error-estimate}.

\subsection{Bounds related to $U_\Lambda (x)$}
\begin{lemma}\label{lem:U_Lambda-o-small-of-M}
In the above notation,
\begin{align}
U_\Lambda(x) \le C \frac{M(x)}{\delta(x)^{\varepsilon/2}},\quad x\in K,\ \delta(x)\ge 1.
\end{align}
\end{lemma}
\begin{proof}
We fix $A\ge1/2$.
For any $x\in K$ we split the integral in \eqref{eq:U-lambda-f} according to the regions from \eqref{eq:Green-f-killed-our-bounds} as follows
\begin{align}
\begin{split}
U_\Lambda (x)&=
\Bigg( \, \,
\int\limits_{
             \begin{subarray}{c}
         y\in K,\,    |x|\le|y|,\\
          |x-y|\ge A|x|
             \end{subarray}}
             +
     \int\limits_{
             \begin{subarray}{c}
          y\in K,\,  |y|\le|x|,\\
          |x-y|\ge A|x|
             \end{subarray}}
             +
        \int \limits_{
             \begin{subarray}{c}
             y\in K\\
             \frac{\delta(x)}{2}\le|x-y|\le A|x|
             \end{subarray}}
        +
        \int\limits_{
        \begin{subarray}{c}
        y\in K\\
        |x-y|\le\frac{\delta(x)}{2}
          \end{subarray}}
            \!\!\!\! \Bigg)
             G(x,y) \Lambda (y) \ud y\\
             &=
             I_1(x) + I_2(x) +I_3(x)+ I_4(x).
\end{split}
\end{align}
We observe that for any $r>0$ and $\theta \in \mathbb{S}^{d-1}$ there is a constant $c>0$ such that
\begin{align}\label{eq:delta-delta-compare}
\delta(r\theta )\le r \delta_\Sigma (\theta) \le c\, \delta (r\theta),
\end{align}
where $\delta_\Sigma (\theta)=\mathrm{dist}(\theta, \Sigma)$ and $\Sigma = K\cap \mathbb{S}^{d-1}$.
\medskip

\noindent \textit{1.\ Estimate of $I_1(x)$}. By using \eqref{eq:Green-f-killed-our-bounds} and \eqref{eq:M-Michalik-up} we obtain
\begin{align}\label{eq:I_2-long-est}
\begin{split}
I_1(x)
&\le C M(x)\!\!\!\!\!
\int\limits_{
             \begin{subarray}{c}
         y\in K,\,    |x|\le|y|,\\
          |x-y|\ge A|x|
             \end{subarray}}
             \frac{M^2(y)}{|y|^{d+2\beta -\alpha}}\frac{\ud y}{(1+\delta (y))^{\alpha +\varepsilon/2}}\\
&\le C_1 M(x)\!\!\!\!\!
\int\limits_{
             \begin{subarray}{c}
         y\in K,\,    |x|\le|y|
             \end{subarray}}
             \frac{\ud y}{|y|^d (1+\delta(y))^{\varepsilon/2}}\\
& \le C_2\, M(x)
	\int_{|x|}^\infty \frac{r^{d-1}}{r^d}\int_{\Sigma}\frac{1}{(1+\delta(r\theta))^{\varepsilon/2}}\, \ud \theta \, \ud r\\
	& \le C_3\, M(x)
	\int_{|x|}^\infty \frac{1}{r^{1+\varepsilon/2}}\, \ud r
\le C_4\, \frac{M(x)}{\delta(x)^{\varepsilon/2}},
\end{split}
\end{align}
where in the penultimate inequality we applied \eqref{eq:delta-delta-compare}.
\medskip

\noindent \textit{2.\ Estimate of $I_2(x)$}. We combine \eqref{eq:Green-f-killed-our-bounds} with \eqref{eq:M-Michalik-up} and this implies
\begin{align*}
I_2 (x) &\le C
        \frac{M(x) }{|x|^{d-\alpha +2\beta}}
\int\limits_{
             \begin{subarray}{c}
          y\in K,\,  |y|\le|x|,\\
          |x-y|\ge A|x|
             \end{subarray}}
    \frac{M^2(y)}{(1+\delta(y))^{\alpha +\varepsilon/2}}\ud y\\
    &\le C_1
            \frac{M(x) }{|x|^{d-\alpha +2\beta}}
\int\limits_{
             \begin{subarray}{c}
          y\in K,\,  |y|\le|x|,\\
          |x-y|\ge A|x|
             \end{subarray}}
    \frac{|y|^{2\beta -\alpha}}{(1+\delta(y))^{\varepsilon/2}}\ud y    \\
    &\le C_2
     \frac{M(x) }{|x|^{d-\alpha +2\beta}}
     \int_{0}^{|x|}r^{d+2\beta -\alpha -1}\int_{\Sigma}\frac{1}{(1+\delta (r\theta))^{\varepsilon/2}}\ud \theta\, \ud r\\
     &\le C_3 \frac{M(x) }{|x|^{d-\alpha +2\beta}}\int_{0}^{|x|} r^{d+2\beta -\alpha -1-\varepsilon/2}\ud r
     \le C_4
     \frac{M(x)}{\delta (x)^{\varepsilon/2}},
\end{align*}
where in the penultimate inequality we again applied \eqref{eq:delta-delta-compare}.
\medskip

\noindent \textit{3.\ Estimate of $I_3(x)$}. We first note that from the inequality $|x-y|\le A|x|$ it follows  that
\begin{align}\label{eq:I_3-help-1}
(1-A)|x| \le|y| \le(1+A)|x|.
\end{align}
Estimate \eqref{eq:Green-f-killed-our-bounds} combined with \eqref{eq:M-Michalik-up} and \eqref{eq:I_3-help-1} yields
\begin{align*}
I_3(x) &\le C
\frac{M(x)}{|x|^{\beta -\alpha /2}}
   \int \limits_{
             \begin{subarray}{c}
             y\in K\\
             \frac{\delta(x)}{2}\le|x-y|\le A|x|
             \end{subarray}}
             \frac{1}{|x-y|^d}
        \frac{M^2(y)}{|y|^{\beta -\alpha/2} (1+\delta (y))^{\alpha +\varepsilon/2}}\, \ud y\\
        &\le C_1
        M(x)    \int \limits_{
             \begin{subarray}{c}
             y\in K\\
             \frac{\delta(x)}{2}\le|x-y|\le A|x|
             \end{subarray}}
    \frac{1}{|x-y|^d}\frac{1}{(1+\delta(y))^{\varepsilon/2}}\ud y
%     & \le C_2\, M(x)
%  \sum_{n=0}^{\lfloor \log_2(A|x|/\delta(x))\rfloor} \!\!\!\! \!\!\!\! \!\!\!\!
%  \int\limits_{
%              \begin{subarray}{c}
%          y\in K\\
%           2^{n-1}\delta (x)\le|x-y|\le2^{n}\delta (x)
%              \end{subarray}}
%  					 \frac{\ud y}{|x-y|^d (1+\delta(y))^{\varepsilon/2}}\\
%  &\le C_3\, M(x)
%  	\sum_{n=0}^{\lfloor \log_2(A|x|/\delta(x))\rfloor} \frac{1}{\left( 2^{n-1}\delta (x)\right)^d}
%  	 \!\!\!\!
%  \int\limits_{
%              \begin{subarray}{c}
%          y\in K\\
%          2^{n-1}\delta (x)\le|x-y|\le2^{n}\delta (x)
%              \end{subarray}}
%               \frac{\ud y}{ (1+\delta(y))^{\varepsilon/2}}\\
% &\le C_4\, M(x)
% 		\sum_{n=0}^{\lfloor \log_2(A|x|/\delta(x))\rfloor} \frac{1}{2^{nd}\delta (x)^d} \left( 2^{n}\delta(x)\right)^{d-1}
% 		\int_{0}^{2^{n}\delta(x)}\int_{\Sigma} \frac{1}{(1+\delta (r\theta))^{\varepsilon/2}}\, \ud \theta\, \ud r\\
% &\le C_5\, \frac{M(x)}{\delta(x)}
% 		\sum_{n=0}^{\lfloor \log_2(A|x|/\delta(x))\rfloor} \!\!\frac{1}{2^n}
% 		\int_{0}^{2^{n}\delta(x)} r^{-\varepsilon/2}\ud r
% \le C_6\, \frac{M(x)}{\delta(x)^{\varepsilon/2}}.
\end{align*}
%and to estimate the last integral we can proceed in the same way as in calculation \eqref{eq:I_2-long-est} but with the annuli $\{2^{n-1}\delta(x)\le|y|\le2^n\delta(x)\}$, for $n=0,1,\ldots ,\lfloor \log_2(A|x|/\delta(x))\rfloor$.
Next, we note that for any $r>0$ the Lebesgue measure of each of the sets
\begin{align*}
A_k (r) =\{ y:\, r\le|x-y|\le r+1,\, k\le\delta(y)\le k+1\},\quad k=0,1,\ldots ,r
\end{align*}
does not exceed $c_0r^{d-2}$, for a constant $c_0>0$ which does not depend neither  on $r$, nor on $k$.
This allows us to estimate the last integral as follows
\begin{align}\label{eq:U_Lambda-lem-help.i3}
\begin{split}
    \int \limits_{
        \begin{subarray}{c}
        y\in K\\
        \frac{\delta(x)}{2}\le|x-y|\le A|x|
        \end{subarray}}
        \frac{\ud y}{|x-y|^d (1+\delta (y))^{\varepsilon/2}}
&\le C
\int_{\frac{1}{2}\delta(x)}^\infty \frac{r^{d-2}}{r^d}\sum_{k=0}^{r+1} \frac{1}{(1+k)^{\varepsilon/2}}\, \ud r\\
&\le C_1
\int_{\frac{1}{2}\delta(x)}^\infty r^{-1-\varepsilon/2}\ud r = C_2\delta(x)^{-\varepsilon/2}.
\end{split}
\end{align}
This results in the bound $I_3(x) \le C M(x)/\delta(x)^{\varepsilon/2}$.

\noindent \textit{4.\ Estimate of $I_4(x)$}. We observe that
\begin{align}\label{eq:deltas-close}
|x-y|\le\frac{\delta(x)}{2}\, \Longrightarrow
\begin{cases}
\frac{|x|}{2}\le|y| \le\frac{3}{2}|x|;\\
\frac{\delta(x)}{2}\le\delta (y) \le\frac{3}{2}\delta(x).
\end{cases}
\end{align}
In view of  equations \eqref{eq:Green-f-killed-our-bounds}, \eqref{eq:M-Michalik-low}-\eqref{eq:M-Michalik-up} and \eqref{eq:deltas-close} we obtain
\begin{align*}
I_4 (x)&\le C
\frac{1}{\delta(x)^{\alpha+\varepsilon/2}}
	\int_{|x-y|\le\frac{\delta(x)}{2}}\frac{M(y)}{|x-y|^{d-\alpha}}\ud y
	\le C_1
\frac{1}{\delta(x)^{\alpha+\varepsilon/2}}
	\int_{|x-y|\le\frac{\delta(x)}{2}}\frac{\delta(y)^{\alpha/2}|y|^{\beta -\alpha/2}}{|x-y|^{d-\alpha}}\ud y\\
	&\le C_2
	\frac{M(x)}{\delta(x)^{\alpha +\varepsilon/2}}\int_{|x-y|\le\frac{\delta(x)}{2}}\frac{\ud y}{|x-y|^{d-\alpha}}
	\le C_3
	\frac{M(x)}{\delta(x)^{\alpha +\varepsilon/2}}\int_{0}^{\frac{\delta(x)}{2}}r^{\alpha -1}\ud r
	\le C_4 \frac{M(x)}{\delta(x)^{\varepsilon/2}}
\end{align*}
and the proof is completed.
\end{proof}

In the next lemma we find the upper bound for the gradient of $U_\Lambda(x)$.
\begin{lemma}\label{gradientestimate}
Assume that $\alpha>1$.
We have
\begin{align}\label{eq:U_Lambda-gradient-est}
\nabla U_\Lambda (x) \le C \frac{M(x)}{\delta (x)^{1+\varepsilon/2}},\quad x\in K,\quad \delta (x)\ge1.
\end{align}
\end{lemma}

\begin{proof}
The proof is simiilar to the proof of Lemma \ref{lem:U_Lambda-o-small-of-M}. For any fixed $A\ge1/2$ we use the decomposition
\begin{align}
\begin{split}
\nabla U_\Lambda (x)&=
\Bigg( \, \,
\int\limits_{
             \begin{subarray}{c}
         y\in K,\,    |x|\le|y|,\\
          |x-y|\ge A|x|
             \end{subarray}}
             +
     \int\limits_{
             \begin{subarray}{c}
          y\in K,\,  |y|\le|x|,\\
          |x-y|\ge A|x|
             \end{subarray}}
             +
        \int \limits_{
             \begin{subarray}{c}
             y\in K\\
             \frac{\delta(x)}{2}\le|x-y|\le A|x|
             \end{subarray}}
        +
        \int\limits_{
        \begin{subarray}{c}
        y\in K\\
        |x-y|\le\frac{\delta(x)}{2}
          \end{subarray}}
            \!\!\!\! \Bigg)
           \nabla  G(x,y) \Lambda (y) \ud y\\
             &=
             I_1(x) + I_2(x) +I_3(x)+ I_4(x).
\end{split}
\end{align}
To estimate integrals $I_i (x)$ for $i=1,2,3$ we observe that in view of \eqref{eq:Green-killed-gradient-est},
\begin{align*}
\nabla G(x,y) \le C \frac{G(x,y)}{\delta (x)},\quad |x-y|\ge\delta (x)/2.
\end{align*}
Using this estimate in conjunction with Lemma \ref{lem:Green-bounds-our} and by repeating the arguments from the proof of Lemma \ref{lem:U_Lambda-o-small-of-M}, we can easily show that
\begin{align*}
I_i (x) \le C \frac{M(x)}{\delta (x)^{1+\varepsilon/2}},\quad i=1,2,3.
\end{align*}
To estimate $I_4(x)$ it is enough to notice that, by \eqref{eq:Green-killed-gradient-est},
\begin{align}\label{eq:Green-gradient-help}
\nabla G(x,y) \le C \frac{1}{|x-y|^{d-\alpha +1}},\quad |x-y|\le\frac{\delta(x)}{2}.
\end{align}
We can again repeat the procedure for $I_4(x)$ from the proof of Lemma \ref{lem:U_Lambda-o-small-of-M} and combine it with \eqref{eq:Green-gradient-help} to obtain
\begin{align*}
I_4(x) \le C \frac{M(x)}{\delta (x)^{1+\varepsilon/2}},
\end{align*}
and this allows us to conclude the required estimate in \eqref{eq:U_Lambda-gradient-est}.
\end{proof}

In the next lemmas we analyse how
$|U_\Lambda (x+y) - U_\Lambda (x)|$ behaves asymptotically (for large $\delta(x)$).
\begin{lemma}\label{lem:U-Lambda-diff-est}
Assume that $0<\alpha <1$. Then there is a constant $C>0$ such that for sufficiently large $\delta (x)$
%\to \infty$
we have
\begin{align}\label{eq:U_Lambda-diff-est}
|U_\Lambda (x+y) - U_\Lambda (x)|\le C \Lambda (x) |y|^{\alpha},\quad |y|\le\frac{\delta (x)}{2}.
\end{align}
\end{lemma}

\begin{proof}
According to \eqref{eq:U-lambda-f} we have
\begin{align}\label{eq:U_Lambda-diff-est-1}
\left\vert U_\Lambda (x+y) - U_\Lambda (x) \right\vert \le \int_{K} \left\vert G(x+y,z) - G(x,z)\right\vert \Lambda(z)\, \ud z.
\end{align}
To obtain the estimate in \eqref{eq:U_Lambda-diff-est} we split the integral in \eqref{eq:U_Lambda-diff-est-1} into five parts and we consider the corresponding cases below.
\medskip

\noindent \textit{1.\ Estimate in $D_1 = \{z\in K:\, |z-x|\le\frac{3}{4}|y|\}$}.
We start by noting that as $|y|\le\delta(x)/2$ and $|z-x|\le3\delta(x)/8$, we have
\[
|z|\le\left(1+\frac{3}{8}\right)|x|,\quad
\delta (z)\ge\delta(x)-|z-x|\ge\frac{5}{8}\delta (x),\quad z\in D_1.
\]
It follows that $\Lambda (z)\le C\Lambda (x)$, for all $z\in D_1$. We bound the killed Green function with the full Green function of the process $Z(t)$ and we obtain
\begin{align*}
\int_{D_1} \left\vert G(x+y,z) - G(x,z)\right\vert \Lambda(z)\, \ud z
&\leq
C\Lambda (x)
\int_{D_1}\left( \frac{1}{|x+y-z|^{d-\alpha}}+ \frac{1}{|x-z|^{d-\alpha}} \right) \ud z.
\end{align*}
It evidently holds $|x+y-z|\ge|y|/4$, for $z\in D_1$. Hence
\begin{align*}
\int_{D_1} \left\vert G(x+y,z) - G(x,z)\right\vert &\Lambda(z)\, \ud z\\
&\leq
C_1 \Lambda (x)
\int_{0}^{3|y|/4}\left(|y|^{ -d+\alpha} +  r^{-d+\alpha}\right)r^{d-1}\ud r
\le C_2\Lambda(x)|y|^\alpha.
\end{align*}
\noindent \textit{2.\ Estimate in $D_2 = \{z\in K:\, |x+y-z|\le\frac{3}{4}|y|\}$}. For $z\in D_2$ we have
$|z|\le(1+7/8)|x|$ and $\delta(z)\ge\delta (x)-|x-z|\ge\delta(x)/8$, and consequently, $\Lambda(z)\le C\Lambda (x)$, for all $z\in D_2$.
Further, $|x-z|\ge|y|/4$, for $z\in D_2$, thus we can argue similarly as in the case of $D_1$ to show that
\begin{align*}
\int_{D_2} \left\vert G(x+y,z) - G(x,z)\right\vert \Lambda(z)\, \ud z
\le C_1\Lambda(x)|y|^\alpha.
\end{align*}

\noindent \textit{3.\ Estimate in $D_3 = \{z\in K:\, |x-z|\le\frac{{3}\delta(x)}{{4}}\}\cap (D_1\cup D_2)^c$}.
We first notice that the Taylor formula ensures that for $z\in D_3$ and $|y|\le\delta(x)/2$,
\begin{align}\label{eq:Green-killed-Taylor}
\left\vert G(x+y,z) - G(x,z)\right\vert \le|y||\nabla_x G(x+\varrho_0y,z)|,\quad \varrho_0 \in (0,1).
\end{align}
Further, for $z\in D_3$ we have $|z|\le{7}|x|/{4}$ and $\delta(x)/{4}\le\delta(z) \le7\delta(x)/4$, and whence $\Lambda (z)\le C\Lambda (x)$. By \eqref{eq:Green-killed-gradient-est} and by bounding the killed Green function with the full Green function we arrive at
\begin{align*}
\int_{D_3} \left\vert G(x+y,z) - G(x,z)\right\vert &\Lambda(z)\, \ud z\\
&\le C \Lambda (x) |y|\int_{D_3}\frac{1}{|x+\varrho_0 y-z|^{d-\alpha}}\frac{1}{\delta(x+\varrho_0 y)\wedge |x+\varrho_0 y-z|}\ud z.
\end{align*}
Next, we observe that for $z\in D_3$
\begin{align*}
|x+\varrho_0 y-z|\geq
\begin{cases}
\frac{1}{3}|x-z|,& \varrho_0\le\frac{1}{2};\\
\frac{1}{3}|x+y-z|,& \varrho_0 >\frac{1}{2}.
\end{cases}
\end{align*}
Further, for $z\in D_3$, $|x+y-z|\le\delta(x)$ and  $\delta (x+\varrho_0 y)\ge\delta(x)/2$, for $|y|\le\delta(x)/2$. These estimates yield
\begin{align*}
\frac{1}{|x+\varrho_0 y-z|^{d-\alpha}}\frac{1}{\delta(x+\varrho_0 y)\wedge |x+\varrho_0 y-z|}
&
\le C
\frac{1}{\left( |x-z|\wedge |x+y-z|\right)^{d-\alpha +1}}\\
&\le C_1
\begin{cases}
\frac{1}{|y|^{d-\alpha +1}},& |x-z|\le3|y|;\\
\frac{1}{|x-z|^{d-\alpha +1}},& |x-z|>3|y|
\end{cases}
\end{align*}
and this allows us to conclude that
\begin{align}\label{eq:int-over-D_3}
\begin{split}
\int_{D_3} &\left\vert G(x+y,z) - G(x,z)\right\vert \Lambda(z)\, \ud z\\
&\le C_1
\Lambda (x) |y|
\left\{ |y|^{\alpha -1} +
\int_{|x-z|>3|y|}
\frac{1}{|x-z|^{d-\alpha +1}}\ud z
\right\}\le C_2 \Lambda (x) |y|^\alpha .
\end{split}
\end{align}

\noindent \textit{4.\ Estimate in $D_4=\{z\in K:\, \frac{3}{4}\delta(x)<|x-z|\le3|x|\}\cap (D_1\cup D_2 \cup D_3)^c$}. We again apply the Taylor formula from \eqref{eq:Green-killed-Taylor} for $z\in D_4$ and we combine it with \eqref{eq:Green-f-killed-bounds} and \eqref{eq:Green-killed-gradient-est} to bound the gradient of $G$ as follows
\begin{align*}
|\nabla_x G(x+\varrho_0 y,z)|
\le C\,
\frac{\delta(x+\varrho_0 y)^{\alpha/2}\delta(z)^{\alpha/2}}{|x+\varrho_0 y-z|^d}
\frac{1}{\delta(x+\varrho_0 y)\wedge |x+\varrho_0 y -z|}.
\end{align*}
Since for $z\in D_4$ it holds $|x+\varrho_0 y-z|\ge\delta(x)/4$ and $\delta (x+\varrho_0 y)\le3\delta(x)/2$, we infer
\begin{align*}
|\nabla_x G(x+\varrho_0 y,z)|
\le C_1\,
\delta(x)^{\alpha/2 -1} \frac{\delta(z)^{\alpha/2}}{|x+\varrho_0 y-z|^d}
\le C_2\,
\delta(x)^{\alpha/2 -1} \frac{\delta(z)^{\alpha/2}}{|x-z|^d},
\end{align*}
where in the last inequality we used the fact that $|x+\varrho_0 y-z|\ge|x-z|/3$, for $z\in D_4$. This and \eqref{eq:M-Michalik-up}-\eqref{eq:M-Michalik-low} imply
\begin{align*}
\int_{D_4} \left\vert G(x+y,z) - G(x,z)\right\vert \Lambda(z)\, \ud z
&\le C_3
\delta(x)^{\alpha/2-1}|y| \int_{D_4} \frac{\delta(z)^{\alpha/2}}{|x-z|^d} \Lambda (z)\ud z\\
&\le C_4\delta(x)^{\alpha/2-1}|y| \int_{D_4} \frac{|z|^{\beta -\alpha/2}}{|x-z|^d (1+\delta (z))^{\varepsilon/2}}\, \ud z \\
&\le C_5\delta(x)^{\alpha/2-1} |x|^{\beta -\alpha/2} |y| \int_{D_4} \frac{\ud z}{|x-z|^d (1+\delta (z))^{\varepsilon/2}} \\
&\le C_6 \frac{M(x)}{\delta (x)} |y| \int_{D_4} \frac{\ud z}{|x-z|^d (1+\delta (z))^{\varepsilon/2}} .
\end{align*}
Next, we note that for any $r>0$ the Lebesgue measure of each of the sets
\begin{align*}
A_k (r) =\{ z:\, r\le|x-z|\le r+1,\, k\le\delta(z)\le k+1\},\quad k=0,1,\ldots ,r
\end{align*}
does not exceed $c_0r^{d-2}$, for a constant $c_0>0$ which does not depend on $r$, nor on $k$.
This allows us to estimate the last integral as follows
\begin{align}\label{eq:U_Lambda-lem-help}
\begin{split}
\int_{D_4} \frac{\ud z}{|x-z|^d (1+\delta (z))^{\varepsilon/2}}
&\le C
\int_{\frac{3}{4}\delta(x)}^\infty \frac{r^{d-2}}{r^d}\sum_{k=0}^{r+1} \frac{1}{(1+k)^{\varepsilon/2}}\, \ud r\\
&\le C_1
\int_{\frac{3}{4}\delta(x)}^\infty r^{-1-\varepsilon/2}\ud r = C_2\delta(x)^{-\varepsilon/2}
\end{split}
\end{align}
and whence
\begin{align}\label{eq:est-in-D_4-for-Lambda-error}
\int_{D_4} \left\vert G(x+y,z) - G(x,z)\right\vert \Lambda(z)\, \ud z
&\le C
\frac{M(x)}{\delta(x)^{1+\varepsilon/2}}|y|\\
&\le\frac{C}{2}\frac{M(x)}{\delta (x)^{\alpha +\varepsilon/2}}|y|^\alpha
\le C_1
\Lambda (x)|y|^\alpha.\nonumber
\end{align}

\noindent \textit{5.\ Estimate in $D_5=\{z\in K:\, |x-z| > 3|x|\}\cap (D_1\cup D_2 \cup D_3 \cup D_4)^c$}. In this region we use the bound from \eqref{eq:Green-f-killed-our-bounds} in the form
\begin{align*}
G(x+\varrho_0 y,z )\le C \frac{M(x+\varrho_0 y)M(z)}{|z|^{d-\alpha
 +2\beta}}
 \le C_1 \frac{M(x)M(z)}{|z|^{d-\alpha
 +2\beta}},\quad z\in D_5,
\end{align*}
which together with \eqref{eq:Green-killed-gradient-est} yields
\begin{align*}
\left\vert \nabla G(x+\varrho_0 y)\right\vert
\le C
\frac{M(x)M(z)}{\delta(x+\varrho_0y)|z|^{d-\alpha
 +2\beta}}
 \le C_1
\frac{M(x)M(z)}{\delta(x+\varrho_0y)|z|^{d-\alpha
 +2\beta}},\quad z\in D_5.
\end{align*}
We thus obtain
\begin{align*}
\int_{D_5} \left\vert G(x+y,z) - G(x,z)\right\vert \Lambda(z)\, \ud z
&\le C
\frac{M(x)}{\delta(x)	}|y| \int_{D_5}
\frac{M^2(z)}{|z|^{d-\alpha +2\beta}(1+\delta(z))^{\alpha +\varepsilon/2}}\, \ud z\\
&\le C_1
\frac{M(x)}{\delta(x)	}|y|
 \int_{D_5} \frac{\ud z}{|z|^{d}(1+\delta(z))^{\varepsilon/2}}.
\end{align*}
To estimate the last integral one can proceed similarly as in \eqref{eq:U_Lambda-lem-help} by using sets $A_k(r)$ and this would imply
\begin{align*}
 \int_{D_5} \frac{\ud z}{|z|^{d}(1+\delta(z))^{\varepsilon/2}}
\le C
\int_{3|x|}^\infty \frac{1}{r^2}\sum_{k=0}^{r+1}(1+k)^{-\varepsilon/2}\ud r\le C_1|x|^{-\varepsilon/2}.
\end{align*}
We finally conclude that
\begin{align}\label{eq:est-in-D_5-for-Lambda-error}
\int_{D_5} \left\vert G(x+y,z) - G(x,z)\right\vert \Lambda(z)\, \ud z
\leq
C\frac{M(x)}{\delta(x)^{1+\varepsilon/2}}|y|
\le C_1
\Lambda (x)|y|^\alpha
\end{align}
and the proof is finished.
\end{proof}

\begin{lemma}\label{lem:U_Lambda-diff-alpha>1}
Assume that $1< \alpha <2$. Then there is a constant $C>0$ such that for sufficiently large $\delta (x)$
%$\delta (x)\to \infty$
it holds
\begin{align}\label{eq:U_Lambda-diff-est-two}
|U_\Lambda (x+y) - U_\Lambda (x) -\nabla U_\Lambda (x)\cdot y|\le C \Lambda (x) |y|^{\alpha},\quad |y|\le\frac{\delta (x)}{2}.
\end{align}
\end{lemma}

\begin{proof}
The proof goes along the same lines as the proof of Lemma \ref{lem:U-Lambda-diff-est} and we briefly sketch the main steps. We start we the upper bound
\begin{multline*}
|U_\Lambda (x+y) - U_\Lambda (x) -\nabla U_\Lambda (x)\cdot y|\\
\leq
\int_K |G(x+y,z) - G(x,z) -y\cdot \nabla_xG(x,z)|\Lambda(z) \ud z
\end{multline*}
and then we split the last integral over the same regions $D_i$, for $i=1,\ldots ,5$. In regions $D_1$ and $D_2$ we first apply the triangle inequality and then we do the same reasoning as before, separately for $ |G(x+y,z) - G(x,z) |$, and for  $|y||\nabla_xG(x+y,z)|$, where for the part involving the gradient we use the estimate from \eqref{eq:Green-killed-gradient-est}.

In the three remaining regions we  repeat the same arguments as in Lemma \ref{lem:U-Lambda-diff-est}, but this time we use the Taylor expansion in the form
\begin{align*}
\left\vert
G(x+y,z) - G(x,z) -y\cdot \nabla_xG(x,z)
\right\vert
\le C|y|^2 \max_{i,j=1,\ldots n}|D^2_{x_ix_j}G(x+\varrho_0 y,z)|,
\end{align*}
which holds for some $\varrho_0\in (0,1)$, and we combine it with the estimate from \eqref{eq:Green-killed-second-deriv-est}.
\end{proof}

%We will show that $\bbE [U_\Lambda (x+X) - U_\Lambda (x+Z)] = o(\Lambda (x))$. To do so we need
%one technical lemma.
\begin{lemma}\label{lem:U-Lambda-Exp-U-est}
There exist two constants $c_1,c_2>0$ such that
\begin{align}\label{eq:error-U-lambda}
\bbE [U_\Lambda (x+Z)] - U_\Lambda(x) \le-c_1\Lambda(x)
+ c_2 \frac{M(x)}{\delta(x)^{\alpha +\varepsilon}},\quad x\in K\text{ with }\delta(x)\ge1.
\end{align}
\end{lemma}

\begin{proof}
Recall  $\tau_x^Z$ defined  in \eqref{Zand tauZ}.
We start by splitting the integral
\begin{align}\label{eq:U-lambda-exp-spli}
\bbE [U_\Lambda (x+Z)]
=
\bbE [U_\Lambda (x+Z); \tau_x^Z \le1] +
\bbE [U_\Lambda (x+Z); \tau_x^Z > 1].
\end{align}
By Lemma \ref{lem:U_Lambda-o-small-of-M} we know that $U_\Lambda (x)\le c M(x)$. This followed by the estimate of the $J$- integral from Lemma \ref{lem:error-estimate} yields
\begin{align}\label{eq:U-lambda-tau-1}
\bbE [U_\Lambda (x+Z); \tau_x^Z \le1]
\le C
\frac{M(x)}{\delta(x)^{\alpha +\varepsilon}}.
\end{align}
Recall that $G(x,y) = G_K (x,y)$. For brevity we use the similar notation for the killed heat kernel $p_K(t,x,y)= p(t,x,y)$ which was defined in \eqref{eq:killed-heat-kernel}.
For the second integral we write
\begin{align*}
\bbE [U_\Lambda (x+Z); \tau_x^Z > 1]
&=
\int_{K} p(1,x,u) U_\Lambda(u)\ud u
=
\int_{K} p(1,x,u) \int_K G(u,y)\Lambda(y)\, \ud y\, \ud u\\
&=
\int_K \Lambda (y) \int_K p(1,x,u)\,  G(u,y)\, \ud u\, \ud y\\
&=
\int_K \Lambda (y) \int_K p(1,x,u)\int_{0}^\infty p(t,u,y)\, \ud t\, \ud u\, \ud y
\\
&=
\int_K \Lambda (y)  \int_{1}^\infty p(t,x,y)\, \ud t\,  \ud y\\
&=
\int_K \Lambda(y) \left( G(x,y) - \int_0^1 p(t,x,y)\, \ud t\right) \ud y\\
&=
U_\Lambda (x)
-
\int_K \left( \int_0^1p(t,x,y)\, \ud t\right) \Lambda (y)\, \ud y.
\end{align*}
Hence, by \eqref{eq:U-lambda-exp-spli} and \eqref{eq:U-lambda-tau-1},
\begin{align}\label{eq:U-lambda-help01}
\bbE [U_\Lambda (x+Z)]
-
U_\Lambda (x)
\leq
-
\int_K \left( \int_0^1p(t,x,y)\, \ud t\right) \Lambda (y)\, \ud y
+
C\frac{M(x)}{\delta(x)^{\alpha + \varepsilon}}
\end{align}
and we are left to show that
\begin{align}\label{eq:U-lambda-help11}
\int_K \left( \int_0^1p(t,x,y)\, \ud t\right) \Lambda (y)\, \ud y
\ge C_1 \Lambda(x).
\end{align}
To do this end, we notice that, by \eqref{eq:M-Michalik-up} and \eqref{eq:M-Michalik-low},
\begin{align*}
\inf_{|y-x|\le\frac{\delta(x)}{2}}\Lambda(y)
\ge C \frac{|x|^{\beta-\alpha/2}\delta(x)^{\alpha/2}}{(1+\delta(x))^{\alpha +\varepsilon/2}}
\ge C_1
\Lambda (x)
\end{align*}
and whence, for all $x\in K $ with $\delta (x)\ge1$,
\begin{align*}
\int_K \left( \int_0^1p(t,x,y)\, \ud t\right) \Lambda (y)\, \ud y
&\geq
C\Lambda (x)
\int_0^1 \bbP \left( \sup_{s\leq t}|Z_s-x|\leq \frac{\delta(x)}{2}\right) \ud t\\
&\geq
C\Lambda (x)
\int_0^1 \bbP \left( \sup_{s\leq t}|Z_s-x|\leq \frac{1}{2}\right) \ud t\geq C_1\Lambda (x).
\end{align*}
Finally, combining \eqref{eq:U-lambda-help01} with \eqref{eq:U-lambda-help11} yield the desired estimate in \eqref{eq:error-U-lambda}.
\end{proof}

\begin{lemma}\label{lem:U-Lambda-Exp-diff-alpha<1}
For $0<\alpha <1$,
\begin{align}
\bbE [U_\Lambda (x+X) - U_\Lambda (x+Z)] = o(\Lambda (x)), \quad \text{as $\delta(x)\to \infty$.}
\end{align}
\end{lemma}

\begin{proof}
For any fixed $R>0$ we write
\begin{align*}
\bbE &[U_\Lambda (x+X) - U_\Lambda (x+Z)]
=
\int_{\bbR^d} \left( U_\Lambda (x+y) - U_\Lambda(x)\right) \left( p_X(y) - p_Z(y)\right) \ud y\\
&=
\left( \int_{|y|\le R} + \int_{R<|y|\le\frac{\delta(x)}{2}} + \int_{|y|>\frac{\delta(x)}{2}}\right)
\left( U_\Lambda (x+y) - U_\Lambda(x)\right) \left( p_X(y) - p_Z(y)\right) \ud y\\
&=I_1(x) + I_2(x) + I_3(x).
\end{align*}
We first estimate $I_2(x)$.  By Lemma \ref{lem:U-Lambda-diff-est} and our assumption \eqref{eq:Ass-densities-diff} we immediately obtain
\begin{align*}
|I_2(x)|
&\leq
C\Lambda (x)
\int_{R<|y|\le\frac{\delta(x)}{2}}
|y|^\alpha
|p_X(y) - p_Z(y)|\ud y
\le C_1
\Lambda (x)
\int_{R<|y|\le\frac{\delta(x)}{2}}
|y|^\alpha
|y|^{-d-\alpha -\nu} \ud y\\
&\leq
C_2\Lambda (x)
\int_{R}^\infty \frac{\ud r}{r^{1+\nu}}
\leq
C_3
\frac{\Lambda(x)}{R^{\nu}}
\end{align*}
and the last expression is equal to $o(\Lambda (x))$ if we choose $R=R(x)\to\infty$. \\
To estimate $I_3(x)$ we use \eqref{eq:Ass-densities-diff} and \eqref{eq:stable-denisty-bound} which yield
\begin{align*}
|I_3(x)|
&\leq
U_\Lambda (x) \int_{|y|>\frac{\delta(x)}{2}}
|p_X(y) - p_Z(y)|\ud y
+
\int_{|y|>\frac{\delta(x)}{2}}
U_\Lambda (x+y)
|p_X(y) - p_Z(y)|\ud y\\
&\leq
C \frac{U_\Lambda(x)}{\delta(x)^{\alpha + \nu}}
+
C_1\frac{1}{\delta(x)^{\nu}}
\int_{|y|>\frac{\delta(x)}{2}}
U_\Lambda (x+y)
p_Z(y)\ud y.
\end{align*}
We next handle the last integral from the estimate above. By Lemma \ref{lem:U_Lambda-o-small-of-M} and Lemma \ref{lem:U-Lambda-Exp-U-est},
\begin{align*}
\int_{|y|>\frac{\delta(x)}{2}}
U_\Lambda (x+y)
p_Z(y)\ud y
&=
\bbE[U_\Lambda(x+Z)]
-
\int_{|y|\le\frac{\delta(x)}{2}}
U_\Lambda (x+y)
p_Z(y)\ud y\\
&=
\bbE[U_\Lambda(x+Z)]  -U_\Lambda(x)
+
U_\Lambda(x)\bbP (|Z_1|> \delta(x/2))\\
&\qquad \qquad \quad -
\int_{|y|\leq\frac{\delta(x)}{2}}
\left( U_\Lambda (x+y) -U_\Lambda(x)\right)
p_Z(y)\ud y\\
&\leq
-c_1\Lambda (x) +c_2\frac{M(x)}{\delta(x)^{\alpha +\varepsilon}} + C U_\Lambda(x)\frac{1}{\delta(x)^{\alpha}}\\
&\qquad \qquad \quad -
\int_{|y|\leq\frac{\delta(x)}{2}}
\left( U_\Lambda (x+y) -U_\Lambda(x)\right)
p_Z(y)\ud y
\end{align*}
and for the last integral we apply Lemma \ref{lem:U-Lambda-diff-est} to get
\begin{align*}
\int_{|y|\leq\frac{\delta(x)}{2}}
\left( U_\Lambda (x+y) -U_\Lambda(x)\right)
p_Z(y)\ud y
&\le C
\Lambda (x)
\int_{|y|\leq\frac{\delta(x)}{2}}
|y|^\alpha p_Z(y)\ud y
\le C_1
\Lambda(x) \log \delta(x).
\end{align*}
It follows that
\begin{align*}
I_3(x)
\leq
C\frac{\Lambda(x)}{\delta(x)^{\nu}} + C_1\frac{\Lambda (x)\log \delta(x)}{\delta(x)^{\nu}}
\le C_2 \frac{\Lambda(x)}{\delta(x)^{\nu /2}}.
\end{align*}
We finally estimate the integral $I_1(x)$. We split this integral into three parts
\begin{align*}
I_1(x)
&=
\int_{|y|\le R} \left(\int_{K_1}+\int_{K_2}+\int_{K_3}\right) \left( G(x+y,z)-G(x,z)\right) \Lambda (z)
 \left( p_X(y) - p_Z(y)\right) \ud z\, \ud y,
\end{align*}
where
\begin{align*}
K_1 = \mathrm{conv}\left( B(x,\varepsilon(x)\delta(x)) \cup B(x+y,\varepsilon(x)\delta(x))\right)
\end{align*}
and
\begin{align*}
K_2 = \left\{z\in K:\, |x-z|\le\frac{\delta(x)}{2}\right\}\cap K_1^c,\quad
K_3=\left\{z\in K:\, |x-z|>\frac{\delta(x)}{2}\right\}\cap K_1^c,
\end{align*}
for some function $\varepsilon(x)$ such that $\varepsilon(x)\delta(x)\to \infty$. The integral over the region $K_2$ can be handled similarly as the integral over the set $D_3$ in the proof of Lemma \ref{lem:U-Lambda-diff-est} (cf.\ \eqref{eq:int-over-D_3}) and this results in the following bound
\begin{align*}
\int_{K_2}
\left| G(x+y,z)-G(x,z)\right| \Lambda (z) \, \ud z
&\leq
C\Lambda(x)|y| \int_{|x-z|>\varepsilon(x)\delta(x)}\frac{\ud z}{|x-z|^{d-\alpha +1}}\\
&\le C_1|y| \frac{\Lambda(x)}{\left( \varepsilon(x)\delta(x)\right)^{1-\alpha}}.
\end{align*}
To estimate the integral over $K_3$ we apply estimates \eqref{eq:est-in-D_4-for-Lambda-error} and \eqref{eq:est-in-D_5-for-Lambda-error} over sets $D_4$ and $D_5$ from Lemma \ref{lem:U-Lambda-diff-est} which yield
\begin{align}\label{eq:K_3-int-alpha<1}
\begin{split}
\int_{K_3} \left| G(x+y,z)-G(x,z)\right| \Lambda (z) \, \ud z
&\leq
\int_{D_4\cup D_5}\left| G(x+y,z)-G(x,z)\right| \Lambda (z) \, \ud z\\
&\le C \frac{M(x)}{|x|^{1+\varepsilon/2}}|y|
\le C \frac{\Lambda(x)}{\delta(x)^{1-\alpha}}|y|.
\end{split}
\end{align}
We next set
$$
\Delta(u,z)=\frac{A_{d,\alpha}}{|u-z|^{d-\alpha}}-G(u,z).
$$
To estimate the integral over $K_1$ we first notice that $\delta(z)\sim\delta(x)$ uniformly in $z\in K_1$. Furthermore, applying \eqref{eq:M-gradient},
one infers that $M(z)\sim M(x)$ uniformly in $z\in K_1$. As a result, $\Lambda(z)\sim\Lambda(x)$
uniformly in $z\in K_1$. Consequently,
\begin{align}
\label{eq:new-attempt}
\nonumber
&\int_{K_1}
\left( G(x+y,z)-G(x,z)\right) \Lambda (z) \, \ud z\\
\nonumber
&\hspace{1cm} =\Lambda(x)\left[
\int_{K_1}
\left( \frac{A_{d,\alpha}}{|x+y-z|^{d-\alpha}} -
\frac{A_{d,\alpha}}{|x-z|^{d-\alpha}} \right)\ud z
-
\int_{K_1}
\left( \Delta(x+y,z)-\Delta(x,z)\right)\ud z \right]\\
&\hspace{2cm}+o\left(\Lambda(x)\int_{K_1}\frac{1}{|x-z|^{d-\alpha}}\ud z\right).
\end{align}
Observe that the first integral in this estimate disappear in view of symmetry of the set $K_1$ and translation invariance of the full Green function.
To estimate the second integral we fix $\varepsilon_0(x)\ge\varepsilon (x)$ tending to zero.
With $T=(\varepsilon_0(x) \delta(x))^\alpha $ and the full heat kernel $p_Z(t,u,z)$ we have
\[\Delta(u,z)=\int_0^T(p_Z(t,u,z)-p^K_Z(t,u,z))\, \ud t +  \int_T^\infty(p_Z(t,u,z)-p^K_Z(t,u,z))\, \ud t.
\]
This implies that
\begin{align*}
\left\vert
\Delta(u,z)-\int_0^T(p_Z(t,u,z)-p^K_Z(t,u,z))\ud t
\right\vert
&\leq
\int_T^\infty p(t,u,z)\ud t\\
&= \int_T^\infty t^{-d/\alpha} p_Z\left(1,0, \frac{z-u}{t^{1/\alpha}}\right)\ud t
\\&\le C T^{1-\frac{d}{\alpha}}=C (\varepsilon_0(x) \delta(x))^{\alpha-d}
\end{align*}
and consequently,
\begin{align*}
\int_{K_1}
\left( \Delta(x+y,z)-\Delta(x,z)\right)\ud z
\leq
C (\varepsilon_0(x) \delta(x))^{\alpha-d}
(\varepsilon(x)\delta (x))^d.
\end{align*}
The last expression may be done arbitrarily small by a proper choice of $\varepsilon_0(x)$.
Since $K_1$ is a subset of the ball centred at $x$ and with radius $(R+\varepsilon(x)\delta(x)$,
$$
\int_{K_1}\frac{1}{|x-z|^{d-\alpha}}\ud z\le C(R+\varepsilon(x)\delta(x))^\alpha.
$$
This implies that if $R(x)\to\infty$ and
$\varepsilon(x)\delta(x)\to\infty$ sufficiently slow, then the last term on the right hand side of \eqref{eq:new-attempt} is $o(\Lambda(x))$.

To finish the proof it remains to notice that integrating over $\{y:|y|\le R(x)\}$ does not destroy convergence to zero provided that $R(x)\to\infty$ sufficiently slow.
\end{proof}

\begin{lemma}\label{alphabiiger1}
Under above assumptions, for $1<\alpha <2$, it holds 
\begin{align}
\bbE [U_\Lambda (x+X) - U_\Lambda (x+Z)] = o(\Lambda (x)), \quad \delta(x)\to \infty.
\end{align}
\end{lemma}

\begin{proof}
For any fixed $R>0$ we write
\begin{align*}
\bbE &[U_\Lambda (x+X) - U_\Lambda (x+Z)]\\
&=
\int_{\bbR^d} \left( U_\Lambda (x+y) - U_\Lambda(x)
-y \nabla U_\Lambda(x)
\right) \left( p_X(y) - p_Z(y)\right) \ud y\\
&\hspace*{1cm}
+\nabla U_\Lambda (x) \int_{\bbR^d} y \left( p_X(y) - p_Z(y)\right) \ud y\\
&=
\left( \int_{|y|\le R} \!\!+ \int_{R<|y|\le\frac{\delta(x)}{2}} \!\!+ \int_{|y|>\frac{\delta(x)}{2}}\right)
\left( U_\Lambda (x+y) - U_\Lambda(x)
 -y \nabla U_\Lambda(x)
\right) \left( p_X(y) - p_Z(y)\right) \ud y\\
&=I_1(x) + I_2(x) + I_3(x),
\end{align*}
where we used the fact that $\bbE [X] =\bbE[Z]=0$.
We first estimate $I_2(x)$.  By Lemma \ref{lem:U_Lambda-diff-alpha>1}  and our assumption \eqref{eq:Ass-densities-diff} we  obtain
\begin{align*}
|I_2(x)|
&\leq
C\Lambda (x)
\int_{R<|y|\le\frac{\delta(x)}{2}}
|y|^\alpha
|p_X(y) - p_Z(y)|\ud y
\le C_1
\Lambda (x)
\int_{R<|y|\le\frac{\delta(x)}{2}}
|y|^\alpha
|y|^{-d-\alpha -\nu} \ud y\\
&\leq
C_2\Lambda (x)
\int_{R}^\infty \frac{\ud r}{r^{1+\nu}}
\leq
C_3
\frac{\Lambda(x)}{R^{\nu}}
\end{align*}
and the last expression is equal to $o(\Lambda (x))$ if we take $R=R(x)\to\infty$ as
$\delta(x)\to\infty$ appropriately large.

To estimate $I_3(x)$ we use \eqref{eq:Ass-densities-diff} and \eqref{eq:stable-denisty-bound} which yield
\begin{align*}
|I_3(x)|
&\leq
U_\Lambda (x) \int_{|y|>\frac{\delta(x)}{2}}
|p_X(y) - p_Z(y)|\ud y
+
|\nabla U_\Lambda (x)|
\int_{|y|>\frac{\delta(x)}{2}}
|y| |p_X(y) - p_Z(y)|\ud y\\
&\hspace*{1cm} +
\int_{|y|>\frac{\delta(x)}{2}}
U_\Lambda (x+y)
|p_X(y) - p_Z(y)|\ud y\\
&\leq
C \frac{U_\Lambda(x)}{\delta(x)^{\alpha + \nu}}
+
C_1 \frac{|\nabla U_\Lambda(x)|}{\delta(x)^{\alpha + \nu -1}}
+
C_2\frac{1}{\delta(x)^{\nu}}
\int_{|y|>\frac{\delta(x)}{2}}
U_\Lambda (x+y)
p_Z(y)\ud y.
\end{align*}
We next handle the last integral from the estimate above.
We have
\begin{align}\label{eq:U-Lambda-alpha>1-I_2-help}
\begin{split}
\int_{|y|>\frac{\delta(x)}{2}}&
U_\Lambda (x+y)
p_Z(y)\ud y
=
\bbE[U_\Lambda(x+Z)]
-
\int_{|y|\le\frac{\delta(x)}{2}}
U_\Lambda (x+y)
p_Z(y)\ud y\\
&=
\bbE[U_\Lambda(x+Z)]  -U_\Lambda(x)
+
U_\Lambda(x)\bbP (|Z_1|> \delta(x/2))\\
&\hspace*{1cm} \quad -
\int_{|y|\leq\frac{\delta(x)}{2}}
\left( U_\Lambda (x+y) -U_\Lambda(x) - y\nabla U_\Lambda (x)\right)
p_Z(y)\ud y\\
& \hspace*{2cm} -
\nabla U_\Lambda (x)
\int_{|y|\leq\frac{\delta(x)}{2}}
y\, p_Z(y)\, \ud y.
\end{split}
\end{align}
Since $\bbE [Z]=0$, we have
\begin{align*}
\left \vert \int_{|y|\leq\frac{\delta(x)}{2}}
y\, p_Z(y)\, \ud y \right \vert
=
\left \vert \int_{|y|\ge\frac{\delta(x)}{2}}
y\, p_Z(y)\, \ud y \right \vert \le C \frac{1}{\delta (x)^{\alpha -1}}.
\end{align*}
%\begin{align*}
%...&\leq
%-c_1\Lambda (x) +c_2\frac{M(x)}{\delta(x)^{\alpha +\varepsilon}} + C U_\Lambda(x)\frac{1}{\delta(x)^{\alpha}}\\
%&\qquad \qquad \quad -
%\int_{|y|\leq\frac{\delta(x)}{2}}
%\left( U_\Lambda (x+y) -U_\Lambda(x)\right)
%p_Z(y)\ud y
%\end{align*}
To estimate the penultimate integral in \eqref{eq:U-Lambda-alpha>1-I_2-help}
 we apply Lemma \ref{lem:U_Lambda-diff-alpha>1} which yield
\begin{multline*}
\int_{|y|\leq\frac{\delta(x)}{2}}
\left( U_\Lambda (x+y) -U_\Lambda(x) -y\nabla U_\Lambda (x)\right)
p_Z(y)\ud y\\
\le C
\Lambda (x)
\int_{|y|\leq\frac{\delta(x)}{2}}
|y|^\alpha p_Z(y)\ud y
\le C_1
\Lambda(x) \log x.
\end{multline*}
By Lemma \ref{lem:U_Lambda-o-small-of-M} and Lemma \ref{lem:U-Lambda-Exp-U-est},
it follows from \eqref{eq:U-Lambda-alpha>1-I_2-help} that
\begin{align*}
|I_3(x)|
\leq
C\frac{\Lambda(x)}{\delta(x)^{\nu /2}}+
C_1\frac{\Lambda(x)}{\delta(x)^{\nu}} + C_2\frac{\Lambda (x)\log x}{\delta(x)^{\nu}}
\le C_3 \frac{\Lambda(x)}{\delta(x)^{\nu /2}}.
\end{align*}
We finally estimate the integral $I_1(x)$. We consider the three following integrals
\begin{align*}
\left(\int_{K_1}+\!\! \int_{K_2}+ \!\!\int_{K_3}\right) \left( G(x+y,z)-G(x,z) - y\cdot \nabla_x G(x,z)\right) \Lambda (z),
 %\left( p_X(y) - p_Z(y)\right) \ud z\, \ud y,
\end{align*}
where
\begin{align*}
K_1 = \mathrm{conv}\left( B(x,\varepsilon(x)\delta(x)) \cup B(x+y,\varepsilon(x)\delta(x))\right)
\end{align*}
and
\begin{align*}
K_2 = \left\{z\in K:\, |x-z|\le\frac{\delta(x)}{2}\right\}\cap K_1^c,\quad
K_3=\left\{z\in K:\, |x-z|>\frac{\delta(x)}{2}\right\},
\end{align*}
for some function $\varepsilon(x)\to 0$ such that $\varepsilon(x)\delta(x)\to \infty$. For the integral over $K_2$ we proceed similarly as in the proof of Lemma \ref{lem:U-Lambda-Exp-diff-alpha<1}. By \eqref{eq:Green-killed-second-deriv-est}, we have
\begin{multline*}
\int_{K_2} \left| G(x+y,z)-G(x,z) - y\nabla_x G(x,z)\right| \Lambda (z)\, \ud z \\
\le C
\Lambda (x) |y|^2 \int_{|x-z|\ge\varepsilon(x)\delta(x)}\frac{\ud z}{|x-z|^{d-\alpha +2}}\\
\le C_1
\Lambda (x) |y|^2 \left( \varepsilon(x)\delta(x)\right)^{\alpha -2}.
\end{multline*}
To use this inequality to estimate the integral $I_1(x)$, we need to integrate the obtained upper bound over all $y$ in the ball of radius $R$. We have
\begin{align*}
\int_{|y|\le R} |y|^2 \left( p_X(y) - p_Z(y)\right)\, \ud y\le C R^{2-\alpha -\nu},
\end{align*}
and thus we can choose $R=R(x) = c\varepsilon(x)\delta (x)$, for a constant $c>0$, so that this part of the integral $I_1(x)$ is equal to $o(\Lambda (x))$.
For the integral over $K_3$ we use similar bound as in \eqref{eq:K_3-int-alpha<1}, but this time we base upon Lemma \ref{lem:U_Lambda-diff-alpha>1} and this implies
\begin{multline*}
\int_{K_3} \left| G(x+y,z)-G(x,z) - y\cdot \nabla_x G(x,z)\right| \Lambda (z)\, \ud z \\
\leq
C \frac{M(x)}{\delta (x)^{2+\nu /2}}|y|^2\le C_1
\Lambda (x) |y|^2\delta (x)^{\alpha -2}
\end{multline*}
and we conclude that in this case the same choice of the function $R(x)$ will force the right rate of decay.

In the last integral over $K_1$ the part involving the difference $G(x+y,z) - G(x,z)$ can be handled in the same way as in the proof of Lemma \ref{lem:U-Lambda-Exp-diff-alpha<1}. It remains to estimate the part involving the gradient, that is
\begin{align*}
\int_{K_1} |\nabla G(x,z)| \Lambda (z) \, \ud z
&\le C
\Lambda (x) \int_{K_1}\frac{\ud z}{|x-z|^{d-\alpha +1}}
\le C_1
\Lambda (x) \delta (x)^{\alpha -1}.
\end{align*}
If we integrate the last upper bound over $y$'s belonging to the ball of radius $R$ and use the fact that $\bbE [Z] = \bbE [X]=0$ combined with \eqref{eq:Ass-densities-diff}, we obtain
\begin{align*}
\left\vert \int_{|y|\le R}y (p_X(y) - p_Z(y))\ud y \right\vert
&=
\left\vert \int_{|y|\ge R}y (p_X(y) - p_Z(y))\ud y \right\vert
\le C R^{1-\alpha -\nu}
\end{align*}
 and we again conclude that the choice $R(x) = \varepsilon (x) \delta (x)$ will suffice to make all the parts of the integral $I_1(x)$ to be $o(\Lambda (x))$, as desired.
\end{proof}
Let
\begin{equation}\label{fLambda}
    f_U(x)=\bbE [U_\Lambda (x+X)] - U_\Lambda(x)
\end{equation}
and recall that $f$ is the error function defined in \eqref{errorfunction}.
\begin{lemma}\label{supermartingale1}
There exists constant  $c_0>0$ and $R>0$
    such that for
    $x\in K$ with  $\delta(x)>R$ the following bound holds
    \[
        f_U(x)+f(x)
        \le-c_0 \Lambda(x).
    \]
\end{lemma}
\begin{proof}
From Lemma \ref{lem:error-estimate} and definition of $\Lambda(x)$ given in \eqref{Lambda}
it follows that there exists sufficiently large $\delta(x)$ such that
\begin{align}\label{fosmallLambda}
f(x)\le C \frac{M(x)}{\delta (x)^{\alpha +\varepsilon}}\le C \frac{\Lambda(x)}{\delta (x)^{\varepsilon/2}}\le\varepsilon \Lambda(x).
\end{align}
Similarly, from  Lemmas \ref{lem:U-Lambda-Exp-U-est}, \ref{lem:U-Lambda-Exp-diff-alpha<1} and  \ref{alphabiiger1}
we can conclude that there exist two constants $c_1,c_2>0$ such that
\begin{align*}
f_U(x)&= \bbE [U_\Lambda (x+Z)] - U_\Lambda(x) + \bbE [U_\Lambda (x+X) - U_\Lambda (x+Z)]\\&\le-c_1\Lambda(x)
+ c_2 \frac{M(x)}{\delta(x)^{\alpha +\varepsilon}} +o(\Lambda (x))=-c_1\Lambda(x) +o(\Lambda (x))\le-c_0\Lambda(x)
\end{align*}
for some $c_0>0$ which completes the proof.
\end{proof}

\subsection{Supermartingale and harmonic function}
We denote
\begin{equation}\label{W}
W(x)=M(x)+U_\Lambda(x),\quad x\in K.
\end{equation}
We fix $x_0\in K$ such that $\mathrm{dist}(x_0,x_0+K)>0$. As the cone is circular, we could evidently choose $x_0=e_d$.
For any $R\geq 0$ and  $c\ge 0$ we set
\begin{align}
\label{Y_n}
\nonumber
Y_n^{(c)}
&=W(x+Rx_0+S(n))
{\rm 1}\{\tau_x>n\}\\
&\hspace{3cm}
+c\sum_{k=0}^{n-1}\Lambda(x+Rx_0+S(k)){\rm 1}\{\tau_x>k\},
\end{align}
and for $c=0$ we use the following simplified notation
\begin{equation}\label{Ynzero}
Y_n=Y_n^{(0)}.
\end{equation}
We denote by $\mathcal{F}_n$ the natural filtration  generated by the random walk $S(n)$.
The key result for our further analysis is the following fact.
\begin{lemma}
\label{lem:submart}
For every $0\le c<c_0$ with $c_0$ specified in Lemma \ref{supermartingale1} there exists $R>0$ such that the sequence
$Y_n^{(c)}$
is a $\mathcal{F}_n$-supermartingale.
\end{lemma}
\begin{proof}
We have
\begin{align*}
&\bbE[Y_n^{(c)}-Y_{n-1}^{(c)}|\mathcal{F}_{n-1}]\\
&\hspace{1cm}=\bbE[W(x+Rx_0+S(n))\mathbf{1}\{\tau_x>n\}
-W(x+Rx_0+S(n-1))\mathbf{1}\{\tau_x>n-1\}|\mathcal{F}_{n-1}]\\
&\hspace{2cm}+c\Lambda(x+Rx_0+S(n-1))\mathbf{1}\{\tau_x>n-1\}.
\end{align*}
Since the functions $M$ and $U_\Lambda$ are nonnegative (and hence $W(x) \ge 0$), then for the error function $f$ defined in \eqref{errorfunction} and
for the function $f_\Lambda$ given in \eqref{fLambda}
we can get the following upper bound
\begin{align*}
&\bbE[Y_n^{(c)}-Y_{n-1}^{(c)}|\mathcal{F}_{n-1}]\\
&\hspace{1cm}\le\mathbf{1}\{\tau_x>n-1\}\bbE[W(x+Rx_0+S(n))
-W(x+Rx_0+S(n-1))|\mathcal{F}_{n-1}]\\
&\hspace{2cm}+c\Lambda(x+Rx_0+S(n-1))\mathbf{1}\{\tau_x>n-1\}\\
&\hspace{1cm}=\big(f(x+Rx_0+S(n-1))+f_U(x+Rx_0+S(n-1))\big)
\mathbf{1}\{\tau_x>n-1\}\\
&\hspace{2cm}+c\Lambda(x+Rx_0+S(n-1))\mathbf{1}\{\tau_x>n-1\}.
\end{align*}
According to Lemma \ref{supermartingale1}
$$
\bbE[Y_n^{(c)}-Y_{n-1}^{(c)}|\mathcal{F}_{n-1}]
\le -\left(c_0-c\right)
\Lambda(x+Rx_0+S(n-1))\mathbf{1}\{\tau_x>n-1\}.
$$
This completes the proof.
\end{proof}
To construct a positive harmonic function
we need
the following auxiliary estimate.
\begin{lemma}
\label{lem:beta_bound}
There exists $C>0$, such that for sufficiently large $R$,
\begin{equation}
\label{eq:beta_1}
\bbE\left[\sum_{k=0}^{\tau_x-1}\Lambda(x+Rx_0+S(k))\right]
\le C\, W(x+Rx_0).
\end{equation}
Furthermore, there exists $\varepsilon(R)\downarrow 0$ as $R\rightarrow+\infty$ such that
\begin{equation}
\label{eq:beta_2}
\bbE\left[\sum_{k=0}^{\tau_x-1}|f(x+Rx_0+S(k))|\right]
\le\varepsilon(R)\, M (x+Rx_0).
\end{equation}
\end{lemma}
\begin{proof}
Let $R$ be sufficiently large and $c$ sufficiently small such that $(Y_n^{(c)})_{n\ge 0}$ defined in \eqref{Y_n}
is a supermartingale.
Then,
$$
\bbE [Y_n^{(c)}]\le Y_0^{(c)}=W(x+Rx_0).
$$
Due to non-negativity of $M$ and $U_\Lambda$ we have
$$
\bbE\left[\sum_{k=0}^{n-1}\Lambda(x+Rx_0+S(k))
\mathbf{1}\{\tau_x>k\}\right]
\le C W(x+Rx_0).
$$
Letting here $n\to\infty$ completes the proof of \eqref{eq:beta_1}.
Estimate \eqref{eq:beta_2} follows from \eqref{eq:beta_1}, Lemma \ref{lem:U_Lambda-o-small-of-M}
and estimate \eqref{fosmallLambda}.
\end{proof}

In the next result we establish the existence of the 'shifted' harmonic
function  $V_R(x)$ for killed random walk on exiting from the cone $K$.
\begin{proposition}\label{harmonicfunction}
For  $R> 0$ sufficiently large, the following function
\begin{equation}\label{defW}
V_R(x)=\lim_{n\to\infty}\bbE[M(x+Rx_0+S(n));\tau_x>n]
\end{equation}
is finite and harmonic for the random walk $S(n)$ killed on exiting from $K$, that is,
\begin{equation}\label{Wharmonic}
    V_R(x)=\bbE\left[V_R(x+S(n));\tau_x>n\right],\quad x\in K,\ n\ge1.
\end{equation}
\end{proposition}
\begin{proof}
We consider the following process
$$
L_n=M(x+Rx_0+S(n\wedge\tau_x))
-\sum_{k=0}^{n\wedge\tau_x-1}f(x+Rx_0+S(k)),\quad n\ge0.
$$
%\begin{align*}
%&\bbE[Z_n-Z_{n-1}|\mathcal{F}_{n-1}]\\
%&\hspace{1cm}
%=\bbE[(Z_n-Z_{n-1})\mathbf{1}\{\tau_x>n-1\}|\mathcal{F}_{n-1}]\\
%&\hspace{1cm}
%=\mathbf{1}\{\tau_x>n-1\}
%\bbE[(u(x+Rx_0+S(n))-u(x+Rx_0+S(n-1)))|\mathcal{F}_{n-1}]\\
%&\hspace{3cm}-\mathbf{1}\{\tau_x>n-1\}f(x+Rx_0+S(n-1)).
%\end{align*}
Recalling that $f(y)=\bbE[M(y+X)]-M(y)$, we can conclude that $L_n$ is a martingale.
By the optional stopping theorem applied to this martingale,
\begin{align*}
M(x+Rx_0)
&=\bbE[L_0]=\bbE[L_n]\\
&=\bbE\left[M(x+Rx_0+S(n))
-\sum_{k=0}^{n-1}f(x+Rx_0+S(k));\tau_x>n\right]\\
&\hspace{1cm}
+\bbE\left[M(x+Rx_0+S(\tau_x))
-\sum_{k=0}^{\tau_x-1}f(x+Rx_0+S(k));\tau_x\le n\right].
\end{align*}
Consequently,
\begin{align}\label{eq:proof.1}
\begin{split}
\bbE[M(x+Rx_0+S(n));\tau_x&>n]\\
&=M(x+Rx_0)-\bbE[M(x+Rx_0+S(\tau_x));\tau_x\le n]\\
&\qquad +\bbE\left[\sum_{k=0}^{\tau_x-1}f(x+Rx_0+S(k));\tau_x\le n\right]\\
&\qquad +\bbE\left[\sum_{k=0}^{n-1}f(x+Rx_0+S(k));\tau_x>n\right].
\end{split}
\end{align}
By the monotone convergence theorem, we obtain
\begin{align*}
\lim_{n\to \infty}
\bbE[M(x+Rx_0+S(\tau_x));\tau_x\le n]
=
\bbE [M(x+Rx_0 +S(\tau_x))].
\end{align*}
Further, by \eqref{eq:beta_2} for $R>0$ large enough the random variable $\sum_{k=0}^{\tau_x-1}f(x+Rx_0+S(k))$ is integrable and whence we can apply the dominated convergence theorem to  conclude that
\begin{equation*}
\lim_{n\to\infty}
\bbE\left[\sum_{k=0}^{\tau_x-1}f(x+Rx_0+S(k));\tau_x\le n\right]
=\bbE\left[\sum_{k=0}^{\tau_x-1}f(x+Rx_0+S(k))\right]
\end{equation*}
and
\begin{align*}
&\limsup_{n\to\infty}\left|\bbE\left[\sum_{k=0}^{n-1}f(x+Rx_0+S(k));\tau_x>n\right]\right|\\
&\hspace{2cm}\le \limsup_{n\to\infty}\bbE\left[\sum_{k=0}^{\tau_x-1}|f(x+Rx_0+S(k))|;\tau_x>n\right]=0.
\end{align*}
We arrive at
\begin{align}\label{eq:proof.5}
\begin{split}
&\lim_{n\to\infty}\bbE[M(x+Rx_0+S(n));\tau_x>n]\\
&
\hspace{5mm}=M(x+Rx_0)-\bbE[M(x+Rx_0+S(\tau_x))]
+\bbE\left[\sum_{k=0}^{\tau_x-1}f(x+Rx_0+S(k))\right].
\end{split}
\end{align}
By \eqref{eq:beta_2}, the limit above is finite and thus we are left to prove that the function
\begin{align*}
V_R(x) = \lim_{n\to\infty}\bbE[M(x+Rx_0+S(n));\tau_x>n]
\end{align*}
is harmonic for the killed random walk.
To do this end, we use the following identity
\begin{multline}\label{eq:bound-for-Leb}
\bbE[M(x+Rx_0+S(n+1));\tau_x>n+1]\\
= \int_{K} \bbP (x+S(1)\in \ud y)\, \bbE [M(y+Rx_0 +S(n)); \tau_x>n].
\end{multline}
Applying \eqref{eq:proof.1} and the fact that the function $M$ is non-negative and bounded by the function $W$ we obtain
\begin{align*}
\bbE[M(y+Rx_0+S(n));\tau_x>n]
&\leq
W(y+Rx_0)
+
\bbE\left[\sum_{k=0}^{\tau_x-1}f(x+Rx_0+S(k))\right]\\
&\leq
C\, W(y+Rx_0)\leq C_1 |y+Rx_0|^\beta,
\end{align*}
where we applied \eqref{eq:beta_1} and the fact that $W(x)\leq C|x|^\beta$. Since $\beta <\alpha$, the $\beta$-moment of $X$ is finite and whence we are allowed to use the dominated convergence theorem in \eqref{eq:bound-for-Leb} which yields
\begin{align*}
V_R(x) = \int_{K} \bbP (x+S(1)\in \ud y)\, V_R(y)\, \ud y
=
\bbE [V_R(x+X); \tau_x >1],
\end{align*}
 and this completes the proof.
\end{proof}

\section{Asymptotic behaviour of $\bbP(\tau_x>n)$}\label{sec:asymptoticstau}
\subsection{Preliminary estimates}
We first prove a uniform upper bound for the tail of $\tau_x$.
We recall that function $W$ is defined in \eqref{W}.
\begin{lemma}
\label{lem:UB}
There exist finite constants $C>0$ and $R>0$ such that
$$
\bbP(\tau_x>n)\le C\frac{W(x+Rx_0)}{n^{\beta/\alpha}},
\quad n\ge1.
$$
\end{lemma}
\begin{proof}
 For every $\varepsilon>0$ one has
\begin{equation}
\label{eq:ub1}
\bbP(\tau_x>n)=
\bbP(\tau_x>n,\delta(x+S(n/2))\ge\varepsilon n^{1/\alpha})
+\bbP(\tau_x>n,\delta(x+S(n/2))<\varepsilon n^{1/\alpha}).
\end{equation}
Set $D_1=\{z\in K: \delta(z)<\varepsilon n^{1/\alpha}\}$.
By the Markov property,
$$
\bbP(\tau_x>n,\delta(x+S(n/2))<\varepsilon n^{1/\alpha})
=\int_{D_1}\bbP(x+S(n/2))\in dz,\tau_x>n/2)\bbP(\tau_z>n/2).
$$
Since $K$ is convex, we can upper bound the tail of $\tau_z$ by the tail of exit time from the half space. Therefore, for every $z\in D_1$ we have the following bound: there exist $n_0=n_0(\varepsilon)$ and an absolute constant $C_1$ such that
$$
\bbP(\tau_z>n/2)\le C_1\varepsilon^{\alpha/2},\quad
z\in D_1,\ n\ge n_0.
$$
This implies that
\begin{equation}
\label{eq:ub2}
\bbP(\tau_x>n,\delta(x+S(n/2))<\varepsilon n^{1/\alpha})
\le C_1\varepsilon^{\alpha/2} \bbP(\tau_x>n/2).
\end{equation}
Set now $D_2=K\setminus D_1$.
It follows from homogeneity \eqref{M-homog} of $M(x)$ and \eqref{eq:M-Michalik-low} that there exists constant $C$ such that
$$
W(z+Rx_0)\ge C M(z)\ge c(\varepsilon n^{1/\alpha})^\beta,
\quad z\in D_2.
$$
Consequently, by the Markov inequality,
\begin{align*}
\bbP(\tau_x>n,\delta(x+S(n/2)) \ge\varepsilon n^{1/\alpha})
\le \frac{1}{c\varepsilon^\beta n^{\beta/\alpha}}
\bbE\left[W(x+Rx_0+S(n/2));\tau_x>n/2\right].
\end{align*}
Recalling that $Y_k=W(x+Rx_0+S(k))\mathbf{1}\{\tau_x>k\}$ defined in \eqref{Ynzero} is a supermartingale (see Lemma \ref{lem:submart}), we obtain
\begin{equation}
\label{eq:ub3}
\bbP(\tau_x>n,\delta(x+S(n/2)\ge\varepsilon n^{1/\alpha})
\le C_2\varepsilon^{-\beta}\frac{W(x+Rx_0)}{n^{\beta/\alpha}}.
\end{equation}
Plugging \eqref{eq:ub2} and \eqref{eq:ub3} into \eqref{eq:ub1}, we have, for all $n\ge n_0$,
$$
\bbP(\tau_x>n)\le
C_2\varepsilon^{-\beta}\frac{W(x+Rx_0)}{n^{\beta/\alpha}}
+C_1\varepsilon^{\alpha/2} \bbP(\tau_x>n/2).
$$
If $\frac{n}{2^{m-1}}\ge n_0$ then we can repeat this procedure
$m-1$ times and get
$$
\bbP(\tau_x>n)\le
C_2\varepsilon^{-\beta}\frac{W(x+Rx_0)}{n^{\beta/\alpha}}
\sum_{j=0}^{m-1}
\left(2^{\beta/\alpha}C_1\varepsilon^{\alpha/2}\right)^j
+(C_1\varepsilon^{\alpha/2})^m\bbP(\tau_x>n/2^m).
$$
Choosing $\varepsilon$ so small that $2^{\beta/\alpha}C_1\varepsilon^{\alpha/2}<1$, we infer that
$$
\bbP(\tau_x>n)\le
\frac{C_2\varepsilon^{-\beta}}{1-2^{\beta/\alpha}C_1\varepsilon^{\alpha/2}}
\frac{W(x+Rx_0)}{n^{\beta/\alpha}}
+(C_1\varepsilon^{\alpha/2})^m.
$$
To bound the second summand on the right hand side we choose
$m$
%as large as possible. This means that $m$ is
such that
$\frac{n}{2^{m-1}}\ge n_0$ and $\frac{n}{2^{m}}< n_0$.
Combining the latter inequality with $C_1\varepsilon^{\alpha/2}\le 2^{-\beta/\alpha}$, we conclude that
$$
(C_1\varepsilon^\alpha)^m\le (2^m)^{-\beta/\alpha}
\le (n_0/n)^{-\beta/\alpha}.
$$
Recalling that $W(x+Rx_0)$ is separated from zero, we get the desired inequality.
\end{proof}
 We next prove a bound for the tail of the expectation of the supermartingale
 \begin{align*}
 Y_k=W(x+Rx_0+S(k))\mathbf{1}\{\tau_x>k\}.
 \end{align*}
\begin{lemma}
\label{lem:ETail}
For every $x\in K$ we have, uniformly in $n$,
$$
\bbE\left[W(x+Rx_0+S(n));
\max_{k\le n}|x+S(k)|>An^{1/\alpha}\right]
\le C(x,R)A^{\beta-\alpha}.
$$
\end{lemma}
\begin{proof}
For brevity, we use the following notation
\begin{align*}
\widehat{S}(j)=\max_{k\le j}|x+S(k)|.
\end{align*}
Fix some $\gamma\in(0,1)$ and split the expectation into two parts:
\begin{align}
\label{eq:et1}
\nonumber
&\bbE\left[Y_n;\max_{k\le n}|x+S(k)|>An^{1/\alpha}\right]
=\bbE\left[Y_n; \widehat{S}(n)>An^{1/\alpha}\right]\\
\nonumber
&\hspace{0.5cm}=\bbE\left[Y_n; \widehat{S}(n)>An^{1/\alpha},\, \widehat{S}(n/2)>A\gamma n^{1/\alpha}\right]\\
&\hspace{1cm}+\bbE\left[Y_n; \widehat{S}(n)>An^{1/\alpha},
\, \widehat{S}(n/2)\le A\gamma n^{1/\alpha}\right].
\end{align}
Using the supermartingale property of $Y_n$, one can easily see that
\begin{align}
\label{eq:et2}
\nonumber
&\bbE\left[Y_n; \widehat{S}(n)>An^{1/\alpha},
\, \widehat{S}(n/2)>A\gamma n^{1/\alpha}\right]\\
\nonumber
&\hspace{0.5cm}\le \bbE\left[Y_n; \widehat{S}(n/2)>A\gamma n^{1/\alpha}\right]\\
&\hspace{1cm}\le \bbE\left[Y_{n/2}; \widehat{S}(n/2)
>2^{1/\alpha}A\gamma (n/2)^{1/\alpha}\right].
\end{align}
It follows easily from the definition of $Y_n$ that
\begin{align*}
&\bbE\left[Y_n; \widehat{S}(n)>An^{1/\alpha},
\, \widehat{S}(n/2)\le A\gamma n^{1/\alpha}\right]\\
&\hspace{0.5cm}\le C\bbE\left[|x+Rx_0+S(n)|^\beta;\tau_x>n,\, \widehat{S}(n)>An^{1/\alpha},\, \widehat{S}(n/2)\le A\gamma n^{1/\alpha}\right].
\end{align*}
On the event $\{\widehat{S}(n)>An^{1/\alpha},\,
\widehat{S}(n/2)\le A\gamma n^{1/\alpha}\}$ we have
$$
\widetilde{S}(n)=\max_{k\in[n/2,n]}|S(k)-S(n/2)|>(1-\gamma)An^{1/\alpha}.
$$
Therefore,
$$
|x+Rx_0+S(n)|\le C\widetilde{S}(n)
$$
and, consequently,
\begin{align*}
&\bbE\left[Y_n; \widehat{S}(n)>An^{1/\alpha},\,
\widehat{S}(n/2)\le A\gamma n^{1/\alpha}\right]\\
&\hspace{1cm}\le C\bbE\left[\widetilde{S}(n)^\beta;\tau_x>n,
\widetilde{S}(n)>(1-\gamma)An^{1/\alpha}\right]\\
&\hspace{1cm}\le C\bbE\left[\widetilde{S}(n)^\beta;\tau_x>n/2,
\widetilde{S}(n)>(1-\gamma)An^{1/\alpha}\right]\\
&\hspace{1cm}= C\bbP(\tau_x>n/2)\bbE\left[\widetilde{S}(n)^\beta;\,
\widetilde{S}(n)>(1-\gamma)An^{1/\alpha}\right].
\end{align*}
Noting that, due to the functional limit theorem for the process $S(nt)/n^{n^{1/\alpha}}$,
$$
\bbE\left[\widetilde{S}(n)^\beta;
\widetilde{S}(n)>(1-\gamma)An^{1/\alpha}\right]
\le Cn^{\beta/\alpha}((1-\gamma)A)^{\beta-\alpha}
$$
and applying Lemma~\ref{lem:UB}, we conclude that
\begin{align}
\label{eq:et3}
\bbE\left[Y_n; \widehat{S}(n)>An^{1/\alpha},\, \widehat{S}(n/2)\le A\gamma n^{1/\alpha}\right]
\le C_1A^{\beta-\alpha}W(x+Rx_0),
\end{align}
where $C_1$ depends on $\gamma$. Plugging \eqref{eq:et2} and
\eqref{eq:et3} into \eqref{eq:et1}, we obtain
\begin{align*}
&\bbE\left[Y_n;\widehat{S}(n)>An^{1/\alpha}\right]\\
&\hspace{1cm}\le \bbE\left[Y_{n/2};
\widehat{S}(n/2)>2^{1/\alpha}A\gamma (n/2)^{1/\alpha}\right]+C\frac{W(x+Rx_0)}{A^{\alpha-\beta}}.
\end{align*}
Iterating this procedure $m$ times, we get
\begin{align*}
&\bbE\left[Y_n; \widehat{S}(n)>An^{1/\alpha}\right]\\
&\hspace{1cm}\le \bbE\left[Y_{n/2^m};
\widehat{S}(n/2^m)>2^{m/\alpha}A\gamma^m (n/2^m)^{1/\alpha}\right]\\
&\hspace{2cm}+C\frac{W(x+Rx_0)}{A^{\alpha-\beta}}
\sum_{j=0}^{m-1}(2^{1/\alpha}\gamma)^{-(\alpha-\beta)j}.
\end{align*}
Take $\gamma<1$ so that $2^{1/\alpha}\gamma>1$. Thus,
\begin{align*}
&\bbE\left[Y_n; \widehat{S}(n)>An^{1/\alpha}\right]\\
&\hspace{1cm}\le \bbE\left[Y_{n/2^m}; \widehat{S}(n/2^m)
>2^{m/\alpha}A\gamma^m (n/2^m)^{1/\alpha}\right]+C\frac{W(x+Rx_0)}{A^{\alpha-\beta}}.
\end{align*}
Using the fact that $Y_k\le C|x+Rx_0+S(k)|^\beta$, we conclude that
\begin{align*}
&\bbE\left[Y_{n/2^m}; \widehat{S}(n/2^m)>2^{m/\alpha}A\gamma^m (n/2^m)^{1/\alpha}\right]\\
&\hspace{1cm}\le C(x,R)
\bbE\left[ \widehat{S}(n/2^m)^\beta;\, \widehat{S}(n/2^m) >2^{m/\alpha}A\gamma^m (n/2^m)^{1/\alpha}\right]\\
&\hspace{1cm}\le C(x,R)
(n/2^m)^{\beta/\alpha}\left(2^{m/\alpha}A\gamma^m\right)^{\beta-\alpha}.
\end{align*}
If we can choose $m=m(n)$ so that $n/2^m\to\infty$ and $(n/2^m)^{\beta/\alpha}\left(2^{m/\alpha}\gamma^m\right)^{\beta-\alpha}\le 1$, then we get the desired bound.
Set $\theta=\frac{\beta}{\alpha\log(2\gamma^{\alpha-\beta})}$.
It is easy then to check that any integer from the interval
$(\theta\log n+2,\theta\log n+4)$ possesses the properties mentioned above.
\end{proof}

\subsection{Asymptotic behaviour of $\bbP(\tau_x>n)$}\label{sec:tau-asymp}%: proof of Theorem \ref{mainthm1}}
\begin{proof}[Proof of Theorem \ref{mainthm1}]
Let $m=m(n)\le n/2$. Later we shall impose further restrictions on this sequence. We split the cone $K$ into three subsets:
\begin{align*}
&D_1=\{z\in K:\delta (z)\le\varepsilon m^{1/\alpha},
|z|\le Am^{1/\alpha}\}, \\
&D_2=\{z\in K:\delta (z)>\varepsilon m^{1/\alpha},
|z|\le Am^{1/\alpha}\},\\
&D_3=\{z\in K:|z|> Am^{1/\alpha}\}.
\end{align*}
By the Markov property at time $m$,
\begin{align}
\label{eq:asymp1}
\nonumber
\bbP(\tau_x>n)
&=\int_K\bbP(x+S(m)\in dz,\tau_x>m)\bbP(\tau_z>n-m)\\
&=\left(\int_{D_1}+\int_{D_2}+\int_{D_3}\right)
\bbP(x+S(m)\in dz,\tau_x>m)\bbP(\tau_z>n-m).
\end{align}
We estimate above integrals separately.
Applying Lemma~\ref{lem:UB}, we have on $D_1$
$$
\bbP(\tau_z>n-m)\le CW(z+Rx_0)n^{-\beta/\alpha}
\le C\varepsilon^{\alpha/2}A^{\beta-\alpha/2}
\frac{m^{\beta/\alpha}}{n^{\beta/\alpha}};
$$
in the last step we have used the bound
$W(z+Rx_0)\le CM(z+Rx_0)$ and \eqref{eq:M-Michalik-low}.
Therefore,
$$
\int_{D_1}\bbP(x+S(m)\in dz,\tau_x>m)\bbP(\tau_z>n-m)
\le C\varepsilon^{\alpha/2}A^{\beta-\alpha/2}
\frac{m^{\beta/\alpha}}{n^{\beta/\alpha}}\bbP(\tau_x>m).
$$
Using once again Lemma~\ref{lem:UB}, we conclude that
for sufficiently large $R$
\begin{align}
\label{eq:asymp2}
\int_{D_1}\bbP(x+S(m)\in dz,\tau_x>m)\bbP(\tau_z>n-m)
\le C\varepsilon^{\alpha/2}A^{\beta-\alpha/2}
\frac{W(x+Rx_0)}{n^{\beta/\alpha}}.
\end{align}

Combining Lemma~\ref{lem:UB} and Lemma~\ref{lem:ETail}, we obtain
\begin{align}
\label{eq:asymp3}
\nonumber
&\int_{D_3}\bbP(x+S(m)\in dz,\tau_x>m)\bbP(\tau_z>n-m)\\
\nonumber
&\hspace{1cm}
\le Cn^{-\beta/\alpha}\int_{D_3}\bbP(x+S(m)\in dz,\tau_x>m)W(z+Rx_0)\\
\nonumber
&\hspace{1cm}=Cn^{-\beta/\alpha}
\bbE[W(x+Rx_0+S(m));\tau_x>m,|x+S(m)|>Am^{1/\alpha}]\\
&\hspace{1cm}\le C(x,R)A^{\beta-\alpha}n^{-\beta/\alpha}.
\end{align}
If we choose $m=m(n)$ so that $\frac{m(n)}{n}\to0$ sufficiently slow then, combining the functional limit theorem for $S(n)$ with Theorem 3.1 in \cite{MR3771748}, we conclude that, for every $R^\prime \geq 0$,
$$\bbP(\tau_z>n-m)\sim
\bbP_{zn^{-1/\alpha}}(\tau^Z_K>1)
\sim\varkappa M(zn^{-1/\alpha})
\sim\varkappa\frac{M(z+R^\prime x_0)}{n^{\beta/\alpha}},\quad \mathrm{as}\ n\to \infty,
$$
uniformly in $z\in D_2$.
Therefore,
\begin{align}
\label{eq:asymp4}
\nonumber
 &\int_{D_2}\bbP(x+S(m)\in dz,\tau_x>m)\bbP(\tau_z>n-m)\\
 &\hspace{1cm}=\frac{\varkappa+o(1)}{n^{\beta/\alpha}}
 \int_{D_2}\bbP(x+S(m)\in dz)M(z+ R^\prime  x_0).
\end{align}
We next notice that
\begin{align*}
&\int_{D_2}\bbP(x+S(m)\in dz,\tau_x>m)M(z+R^\prime  x_0)\\
&\hspace{0.5cm}=\bbE[M(x+ R^\prime  x_0+S(m));\tau_x>m]\\
&\hspace{1cm}-\int_{D_1}\bbP(x+S(m)\in dz,\tau_x>m)M(z+ R^\prime  x_0)\\
&\hspace{1.5cm}-\int_{D_3}\bbP(x+S(m)\in dz,\tau_x>m)M(z+ R^\prime  x_0).
\end{align*}
Recalling that $M(z+R^\prime x_0)\le M(z+Rx_0)\le W(z+Rx_0)$ for all $z\in K$, $R^\prime\le R$ and applying \eqref{eq:asymp3}, we conclude that
\begin{align}
\label{eq:asymp5}
\int_{D_3}\bbP(x+S(m)\in dz,\tau_x>m)M(z+ R^\prime  x_0)
\le C(x, R  )A^{\beta-\alpha}.
\end{align}
Furthermore, using \eqref{eq:M-Michalik-up} and applying then Lemma~\ref{lem:UB}, for sufficiently large $R$ we have
\begin{align}
\label{eq:asymp6}
\int_{D_1}\bbP(x+S(m)\in dz,\tau_x>m)M(z+R^\prime x_0)
\le C\varepsilon^{\alpha/2}A^{\beta-\alpha/2}W(x+R x_0).
\end{align}
Combining \eqref{eq:asymp1}---\eqref{eq:asymp6} we conclude that, for sufficiently large $R$ and for all $0\leq R^\prime  \leq R$,
\begin{align*}
&\left|n^{\beta/\alpha}\bbP(\tau_x>n)
-(\varkappa+o(1))\bbE[M(x+R^\prime x_0+S(m));\tau_x>m]\right|\\
&\hspace{1cm}\le C(x,R)A^{\beta-\alpha}
+C\varepsilon^{\alpha/2}A^{\beta-\alpha/2}W(x+Rx_0).
\end{align*}
Letting first $\varepsilon\to0$ and then $A\to\infty$, we conclude that, for all $0\leq R^\prime  \leq R$,
\begin{align}
\label{eq:interm.step}
 n^{\beta/\alpha}\bbP(\tau_x>n)
=(\varkappa+o(1))\bbE[M(x+R^\prime x_0+S(m));\tau_x>m]+o(1).
\end{align}

Taking first $R^\prime=R$ and using fact that $\bbE[M(x+Rx_0+S(m));\tau_x>m]$ converges to $V_R(x)$
(see Proposition \ref{harmonicfunction}), we have
$$
n^{\beta/\alpha}\bbP(\tau_x>n)\to V_R(x),\quad \mathrm{as}\ n\to \infty.
$$
Comparing this with \eqref{eq:interm.step}, we conclude that, for all $0\leq R^\prime  \leq R$,
$$
\bbE[M(x+  R^\prime x_0+S(m));\tau_x>m]\to V_R(x),\quad \mathrm{as}\ m\to \infty.
$$
This implies that
$$
V(x)=\lim_{n\to\infty}\bbE[M(x+S(n));\tau_x>n]
=V_R(x)
$$
and the limit does not depend on $R$.
In particular, by Proposition \ref{harmonicfunction} the function $V(x)$ is harmonic and
$$n^{\beta/\alpha}\bbP(\tau_x>n)\to V(x),\quad \mathrm{as}\ n\to \infty,
$$
for every fixed $x$.
Thus, the proof is complete.
\end{proof}

\section{Conditional  functional limit theorem}\label{sec:copnditionallimit}
\subsection{Existence of the meander.}\label{existencemeander}
Let $\widehat{\bbP}$ denote the distribution of the Doob $h$-transform of the process $Z$ killed at leaving $K$. It was shown in \cite[Theorem 3.3]{MR4415390} that $\widehat{\bbP}_x$ converges, as $x\to0$, to a probability measure $\widehat{\bbP}_0$ in the Skorohod space. We now transfer this into convergence of conditional distributions.

Let $g$ be a bounded continuous functional on $D[0,1$]. Without loss of generality we may assume that $g$ takes nonnegative values only.

It is immediate from the definition of the Doob transform that
$$
\bbE_x[g(Z);\tau^Z_K>1]
=M(x)\widehat{\bbE}_x\left[\frac{g(Z)}{M(Z(1))}\right].
$$
Consequently,
$$
\bbE_x[g(Z)\mid \tau^Z_K>1]
=\frac{M(x)}{\bbP_x(\tau^Z_K>1)}
\widehat{\bbE}_x\left[\frac{g(Z)}{M(Z(1))}\right].
$$
Applying Theorem 3.1 from \cite{MR3771748}, we have
\begin{equation*}\label{limitmeander}
\bbE_x[g(Z) \mid \tau^Z_K>1]
\sim\varkappa^{-1}
\widehat{\bbE}_x\left[\frac{g(Z)}{M(Z(1))}\right],
\quad\text{as }x\to0.
\end{equation*}
Thus, it remains to show that $\widehat{\bbE}_x\left[\frac{g(Z)}{M(Z(1))}\right]$ converges.
Since the functional $\frac{g(Z)}{M(Z(1))}$ is not bounded we cannot apply the convergence of $\widehat{\bbP}_x$ directly. Fix some $\varepsilon>0$. Then, combining the weak convergence of $\widehat{\bbP}_x$ with the absolute continuity of $Z(1)$ under $\widehat{\bbP}_0$, we conclude that
$$
\widehat{\bbE}_x\left[\frac{g(Z)}{M(Z(1))}; |Z(1)|>\varepsilon\right]\to
\widehat{\bbE}_0\left[\frac{g(Z)}{M(Z(1))}; |Z(1)|>\varepsilon\right],\quad \mathrm{as}\ x\to 0.
$$
Furthermore, by Theorem 3.3 from \cite{MR3771748},
\begin{align*}
\widehat{\bbE}_x\left[\frac{g(Z)}{M(Z(1))}; |Z(1)|\le\varepsilon\right]
&=\frac{1}{M(x)}\bbE_x[g(Z);|Z(1)|\le\varepsilon,\tau^Z_K>1]\\
&\le \frac{C_g}{M(x)}\bbP_x(|Z(1)|\le\varepsilon,\tau^Z_K>1)\\
&\to C_g\varkappa\int_{|y|\le\varepsilon}n_1(y) \ud y,\quad\text{as }x\to0,
\end{align*}
where the limit
  \begin{equation*}
    \label{e:blimit}
    n_1(y) = \lim_{K \ni x \to 0} \frac{p_K(1,x,y)}{\bbP(\tau_x > 1)}
  \end{equation*}
exists and it is a finite and continuous in $y$
for the heat kernel $p_K(t,x,y)$ of the process $Z$ killed on
leaving cone $K$ and defined formally in \eqref{eq:killed-heat-kernel}.

Moreover, by the definition of $\widehat{\bbP}_0$,
\begin{align*}
\widehat{\bbE}_0\left[\frac{g(Z)}{M(Z(1))}; |Z(1)|\le\varepsilon\right]
\le C_g \varkappa\int_{|y|\le\varepsilon}n_1(y)\ud y.
\end{align*}
Using (3.8) from \cite{MR3771748} it is easy to see that $\int_{|y|\le\varepsilon}n_1(y)dy$ converges to zero as $\varepsilon\to0$. This gives the desired convergence:
$$
\bbE_x[g(Z) \mid \tau^Z_K>1]
\to \varkappa^{-1}\widehat{\bbE}_0\left[\frac{g(Z)}{M(Z(1))}\right],\quad \text{as }x\to0,
$$
for every continuous bounded $g$.
Having above limiting law well-defined,
we can define the corresponding stochastic process $me_K(t)$ via Kolmogorov Existence Theorem.

\subsection{Proof of Theorem \ref{mainthm2}.}
As we have just proved the existence of the meander $me_K(t)$ we are left with proving the weak convergence.
Fix some $\eta>0$ and consider a bounded continuous functional $\varphi$ on the space $D[\eta,1]$. Without loss of generality we may assume that $\varphi$ takes values in the interval $[0,1]$.
Let $m=m(n)$ be the sequence which has been used in Subsection \ref{sec:tau-asymp}. By the Markov property
at time $m$,
\begin{align}
\label{eq:flt1}
\nonumber
&\bbE\left[\varphi\left(\frac{x+S(nt)}{n^{1/\alpha}}\right);\tau_x>n\right]\\
&=\int_K\bbP(x+S(m)\in dz,\tau_x>m)
\bbE\left[\varphi\left(\frac{z+S((nt-m)^+)}{n^{1/\alpha}}\right);\tau_z>n-m\right].
\end{align}
We shall use the same decomposition as in Subsection \ref{sec:tau-asymp}. Since $\varphi$ is bounded by $1$, we can use
\eqref{eq:asymp2} and \eqref{eq:asymp3} to get
\begin{align}
\label{eq:flt2}
\nonumber
&\int_{D_1}\bbP(x+S(m)\in dz,\tau_x>m)
\bbE\left[\varphi\left(\frac{z+S((nt-m)^+)}{n^{1/\alpha}}\right);\tau_z>n-m\right]\\
\nonumber
&\hspace{1cm}\le \int_{D_1}\bbP(x+S(m)\in dz,\tau_x>m)\bbP(\tau_z>n-m)\\
&\hspace{1cm}
\le C\varepsilon^{\alpha/2}A^{\beta-\alpha/2}
\frac{W(x+Rx_0)}{n^{\beta/\alpha}}
\end{align}
and
\begin{align}
\label{eq:flt3}
\nonumber
&\int_{D_3}\bbP(x+S(m)\in dz,\tau_x>m)
\bbE\left[\varphi\left(\frac{z+S((nt-m)^+)}{n^{1/\alpha}}\right);\tau_z>n-m\right]\\
\nonumber
&\hspace{1cm}\le
\int_{D_3}\bbP(x+S(m)\in dz,\tau_x>m)\bbP(\tau_z>n-m)\\
&\hspace{1cm}\le C(x,R)A^{\beta-\alpha}n^{-\beta/\alpha}.
\end{align}
On the set $D_2$ we may apply the functional limit theorem for $S(nt)/n^{1/\alpha}$ to conclude that
\begin{align*}
 \bbE\left[\varphi\left(\frac{z+S((nt-m)^+)}{n^{1/\alpha}}\right);\tau_z>n-m\right]
 =(1+o(1))\bbE_{z/n^{1/\alpha}}\left[\varphi(Z);\tau>1\right]
\end{align*}
uniformly in $z\in D_2$.
Using the aruments from the previous subsection, we infer that, uniformly in $z\in D_2$,
$$
\bbE_{z/n^{1/\alpha}}\left[\varphi(Z)|\tau>1\right]
\sim \bbE_0[\varphi(me_K)].
$$
As a result we have
\begin{align*}
 \bbE\left[\varphi\left(\frac{z+S((nt-m)^+)}{n^{1/\alpha}}\right);\tau_z>n-m\right]
 =(\varkappa+o(1))\bbE_0[\varphi(me_K)]\frac{M(z)}{n^{\beta/\alpha}}.
\end{align*}
Consequently,
\begin{align}
\label{eq:flt4}
\nonumber
&\int_{D_2}\bbP(x+S(m)\in dz,\tau_x>m)\,
\bbE\left[\varphi\left(\frac{z+S((nt-m)^+)}{n^{1/\alpha}}\right);\tau_z>n-m\right]\\
&\hspace{1cm}
=(\varkappa+o(1))\bbE_0[\varphi(me_K)]\, n^{-\beta/\alpha}
\int_{D_2}\bbP(x+S(m)\in dz,\tau_x>m)M(x+Rx_0).
\end{align}
Combining \eqref{eq:flt1}---\eqref{eq:flt4} and repeating the arguments from the previous subsection, we conclude that
\begin{align*}
&\bbE\left[\varphi\left(\frac{x+S(nt)}{n^{1/\alpha}}\right);\tau_x>n\right]=
(\varkappa+o(1))V_R(x)\bbE_0[\varphi(me_K)]\, n^{-\beta/\alpha}.
\end{align*}
Clearly, this gives the convergence of conditional expectations.
Thus, the proof of the functional convergence on $D[\eta,1]$ is complete. Since we have this for every $\eta>0$, we have convergence of finite dimensional distribution on the whole interval $[0,1]$ and the tightness on every interval $[\eta,1]$.
Thus, it remains to prove the tightness at zero. But this property will follow if we show that
\begin{equation}
\label{eq:flt5}
\bbP\left(\widehat{S}(n)>An^{1/\alpha}\mid \tau_x>n\right)\le \frac{C}{A^{\alpha-\beta}},
\end{equation}
where $\widehat{S}(n)=\max_{k\le n}|x+S(k)|$.

Similar to the proof of Lemma~\ref{lem:ETail}, we have for every $\gamma<1$ the inequality
\begin{align}
\label{eq:flt6}
\nonumber
&\bbP\left(\widehat{S}(n)>An^{1/\alpha},\tau_x>n\right)\\
\nonumber
&\hspace{1cm}=\bbP\left(\widehat{S}(n)>An^{1/\alpha},
\widehat{S}(n/2)>\gamma An^{1/\alpha},\tau_x>n\right)\\
\nonumber
&\hspace{2cm}+\bbP\left(\widehat{S}(n)>An^{1/\alpha},\widehat{S}(n/2)\le\gamma An^{1/\alpha},\tau_x>n\right)\\
\nonumber
&\hspace{1cm}\le\bbP\left(\widehat{S}(n/2)>\gamma An^{1/\alpha},\tau_x>n\right)\\
&\hspace{2cm}+\bbP\left(\max_{k\in[n/2,n]}|S(k)-S(n/2)|>(1-\gamma)An^{1/\alpha},\tau_x>n/2\right).
\end{align}
Applying Lemma~\ref{lem:UB} and the asymptotic stability of $S(n)$
we have
\begin{align}
\label{eq:flt7}
\nonumber
&\bbP\left( \max_{k\in[n/2,n]}|S(k)-S(n/2)|>(1-\gamma)An^{1/\alpha},\tau_x>n/2\right)\\
\nonumber
&\hspace{1cm}=\bbP\left(\max_{k\in[n/2,n]}|S(k)-S(n/2)|>(1-\gamma)An^{1/\alpha}\right)\bbP(\tau_x>n/2)\\
&\hspace{1cm}\le \frac{C}{A^{\alpha}}\frac{W(x+Rx_0)}{n^{\beta/\alpha}}.
\end{align}
Furthermore, using the Markov property and applying once again Lemma~\ref{lem:UB}, we have
\begin{align*}
&\bbP\left( \widehat{S}(n/2)>\gamma An^{1/\alpha},\tau_x>n\right) \\
&\hspace{1cm}\le\frac{C}{n^{\beta/\alpha}}
\bbE\left[W(x+Rx_0+S(n/2));\widehat{S}(n/2)>\gamma An^{1/\alpha}\right].
\end{align*}
Taking into account Lemma~\ref{lem:ETail}, we finally get
\begin{align}
\label{eq:flt8}
\bbP\left(\widehat{S}(n/2)>\gamma An^{1/\alpha},\tau_x>n\right)
\le \frac{C}{A^{\alpha-\beta}}\frac{W(x+Rx_0)}{n^{\beta/\alpha}}.
\end{align}
Plugging \eqref{eq:flt7} and \eqref{eq:flt8} into \eqref{eq:flt6} and recalling that
$\bbP(\tau_x>n)\sim\varkappa V(x)n^{-\beta/\alpha}$, we obtain \eqref{eq:flt5}.

\bibliographystyle{abbrv}
\bibliography{stable-rw-cones-biblio}

\end{document}